\documentclass[final,12pt]{article}
\usepackage[utf8]{inputenc}
\usepackage{anysize} 
\marginsize{2cm}{2cm}{2cm}{2cm} 
\usepackage{fancyhdr} 
\usepackage{lipsum} 
\usepackage{color} 
\usepackage{bm} 
\usepackage{amsfonts,amsthm,amsmath,amssymb} 
\usepackage{latexsym,graphicx,epstopdf} 
\usepackage{etoolbox} 
\usepackage{hyperref} 
\usepackage[left,mathlines]{lineno} 
\renewcommand{\a}{{\bf a}}
\renewcommand{\b}{{\bf b}}

\newcommand{\ee}{{\rm e}}
\newcommand{\f}{{\bf f}}

\newcommand{\n}{{\bf n}}

\renewcommand{\u}{{\bf u}}
\renewcommand{\v}{{\bf v}}
\newcommand{\w}{{\bf w}}
\newcommand{\x}{{\bf x}}
\newcommand{\y}{{\bf y}}
\newcommand{\z}{{\bf z}}

\newcommand{\zero}{{\bf 0}}

\newcommand{\bb}{{\bf B}}

\renewcommand{\ee}{{\bf E}}

\newcommand{\hh}{{\bf H}}
\newcommand{\ii}{{\bf I}}

\renewcommand{\ll}{{\bf L}}
\newcommand{\mm}{{\bf M}}

\newcommand{\pp}{{\bf P}}

\newcommand{\rr}{{\bf R}}

\newcommand{\ww}{{\bf W}}

\newcommand{\ccc}{\mathbb{C}}

\newcommand{\hhh}{\mathbb{H}}

\renewcommand{\lll}{\mathbb{L}}

\newcommand{\nnn}{\mathbb{N}}

\newcommand{\rrr}{\mathbb{R}}

\newcommand{\oooo}{\mathcal{O}}

\newcommand{\rrrr}{\mathcal{R}}
\newcommand{\tttt}{\mathcal{T}}

\newcommand{\btau}{{\bm\tau}}

\newcommand{\bsigma}{{\bm\sigma}}
\newcommand{\bepsilon}{{\bm\epsilon}}

\renewcommand{\div}{{\rm div}}
\newcommand{\bdiv}{{\bf div}}
\newcommand{\curl}{{\rm curl}}

\newcommand{\tr}{{\rm tr}}
\newcommand{\transpose}{\texttt{t}}

\newtheorem{theorem}{Theorem}
\newtheorem{definition}{Definition}
\newtheorem{lemma}{Lemma}

\usepackage[perpage]{footmisc}

\title{Steklov eigenvalues for the Lam\'e operator in linear elasticity\thanks{This work was partially supported by CONICYT-Chile, through Becas 
Chile, and NSERC through the Discovery program of Canada.}}

\author{Sebasti\'an Dom\'inguez\thanks{Department of Mathematics, Simon Fraser University, Burnaby, BC, 
Canada}{\,\,\,\thanks{Corresponding author: \href{mailto:domingue@sfu.ca}{domingue@sfu.ca}}}}

\begin{document}

\maketitle

\begin{abstract}
In this paper we study Steklov eigenvalues for the Lam\'e operator which arise in the theory of linear elasticity. In this eigenproblem the spectral parameter appears in a Robin boundary condition, linking the traction and the displacement. To establish the existence of a countable spectrum for this problem, we present an extension of Korn's inequality. We also show that a proposed conforming Galerkin scheme provides convergent approximations to the true eigenvalues. A standard finite element method is used to conduct numerical experiments on 2D and 3D domains to support our theoretical findings.
\end{abstract}


\vspace{.25cm}


\section{Introduction}\label{section:intro}
The Steklov problem for the Laplace operator is well-studied in the mathematical community. This eigenproblem was first introduced by V. Steklov in \cite{ref:steklov1902}, and has become a rich source of interesting research. These eigenfrequencies naturally arise in the study of the sloshing phenomenon in fluid mechanics (see, e.g. \cite{ref:mayer2012}). The Steklov spectrum coincides with that of the Dirichlet-to-Neumann map for the Laplacian, see e.g. \cite{ref:girouard2017}. 

The study of the Dirichlet-to-Neumann map for linear elasticity is important in elastostatic problems. In this paper we are interested in the study of Steklov eigenvalues for the Lam\'e operator. More precisely, let us assume that an isotropic and linearly elastic material occupies the region $\Omega$ in $\rrr^d$, $d\geq2$.  We seek non-zero displacements $\u$ of $\Omega$ and frequencies $w\in\ccc$ satisfying the following eigenproblem:
\begin{align}
-\bdiv\,\bsigma(\u) = \zero\,\,\mbox{in $\Omega$},\,\,\bsigma(\u)\n = 
w\,p\,\u\,\,\mbox{on $\partial\Omega$},\label{eq:steklov-lame}
\end{align}
where $\n$ is the outer unit normal vector on $\partial\Omega$, and $\bsigma$ is the Cauchy tensor, defined as
\begin{align}
    \bsigma(\u) := 2\mu\bepsilon(\u) + \lambda\tr(\bepsilon(\u))\ii,\,\,\bepsilon(\u) := \frac{1}{2}(\nabla\u+\nabla\u^\transpose).
\end{align}
The parameters $\lambda\in\rrr$ and $\mu>0$ are the so-called Lam\'e parameters, assumed to satisfy the condition
\begin{align}
\lambda+\left(\frac{2}{d}\right)\mu >0,\label{eq:intro-lame-parameters}
\end{align}
while the parameter $p \in L^\infty(\partial\Omega)$ is assumed to be strictly positive on $\partial\Omega$. Eigenpairs solving \eqref{eq:steklov-lame} are called {\it Steklov-Lam\'e eigenpairs}. 

Historically researchers have studied the case where the displacement is set to zero on a subset of the boundary and the Robin boundary condition in \eqref{eq:steklov-lame} on the remainder. For the existence of the countable spectrum of these problems the standard Korn's inequality suffices \cite{ref:cao2013,ref:gomez2018,ref:ionescu1996,ref:ionescu2005}. The contribution of this paper is to the situations in which there is no Dirichlet part on the boundary. In the physics literature, the parameter $p$ can be thought as the density of the material in $\Omega$ which is concentrated on $\partial\Omega$. For instance, as noted by \cite{ref:hinton1990}, the author in \cite{ref:atkinson1964} discusses 1D boundary value problems where the spectral parameter appears on the boundary conditions. One of the applications the author describes is a long 1D linearly elastic string, with a point mass attached to one of the end points of the string \cite[p. 22]{ref:atkinson1964}.

The first goal of this manuscript is to establish the existence of a countable spectrum of the Steklov-Lam\'e eigenproblem in \eqref{eq:steklov-lame}. To this end, a Korn's-type inequality \cite{ref:korn1906,ref:korn1909} is proved to achieve the existence of a point spectrum. The Sobolev embedding theorem plays an important role in the proof of Korn's inequality. The version of Korn's inequality that we show represents a natural extension to the inequality presented in \cite[Theorem 3.1]{ref:damlamian2018}.
In addition, we are able to show, based on the work in \cite{ref:duran2004,ref:duran2006}, that this inequality is valid over Jones domains (sometimes called $(\epsilon,\delta)$-domains) and John domains. Jones domains were first introduced in \cite{ref:jones1981}, whereas John domains were introduced in \cite{ref:john1961}. These classes of domains constitute a very large class of domains which contain, for example, Lipschitz as well as star-shape domains.

The next goal of this manuscript is to derive suitable numerical schemes to approximate the Steklov-Lam\'e eigenpairs. Following the theory developed in \cite{ref:babuskaosborn1991} for self-adjoint, compact, linear and bounded operators, we are able to show that, under some assumptions, any conforming Galerkin scheme provides a stable approximation to the true eigenpairs. A number of numerical examples are provided with the use of Lagrange elements in 2D and 3D.

The rest of the paper is organized as follows. A proof of the necessary version of Korn's inequality is presented in \ref{section:korns}. Then in \ref{section:steklovlame} we study the existence of a countable spectrum of the Steklov-Lam\'e eigenproblem. 
In \ref{section:discrete} we propose a discrete scheme to approximate these eigenpairs and provide a spectral characterization of the corresponding discrete solution operator. Numerical examples are also presented in \ref{section:discrete}.

\section{A variant of Korn's inequality}\label{section:korns}
We first fix some notation. Vector fields will be denoted with bold symbols whereas tensor fields are denoted with bold Greek letters.
%
The notation $\a\cdot\b$ is the standard dot product with induced norm $\|\cdot\|$. For tensors $\bsigma,\,\btau\in\rrr^{d\times d}$, the double dot product notation $\bsigma:\btau := \tr(\btau^\transpose\bsigma)$ is the usual inner product for tensors, where $\tr(\cdot)$ denotes the trace of a tensor (sum of the main diagonal). This inner product induces the Frobenius norm for tensors, also denoted as $\|\cdot\|$. 

Given a real vector space $H$ of scalar 
fields, we denote by $\hh$ the vector valued functions such that each scalar component belongs to $H$. Further, $\hhh$ is
utilized to denote tensor fields such that every entry belongs to $H$. For an open domain $U$ of $\rrr^d$,
$d\geq 2$, the space $W^{s,p}(U)$ denotes the usual Sobolev space of scalar fields, for $s\in\rrr$ and $1<p<\infty$, with norm $\|\cdot\|_{s,p, U }$. For vector fields, we use the notation $\ww^{s,p}( U )$ with the corresponding norm simply denoted by $\|\cdot\|_{s,p}$. In particular, the Hilbert space $H^s(U)$ is the usual Sobolev space $W^{s,2}( U )$ with norm $\|\cdot\|_{s, U } := \|\cdot\|_{s,2, U }$.
The inner product in $H^s( U )$ is $(\cdot,\cdot)_{s, U }$, whereas $[\cdot,\cdot]_{H^s( U )}$ is the duality pairing between the dual space $\big(H^s( U )\big)^*$ and $H^s( U )$. For vector fields whose entries belong to $H^s(U)$, we use $\hh^s(U)$ with corresponding inner product and norm also denoted by $(\cdot,\cdot)_{s,U}$ and $\|\cdot\|_{s,U}$ respectively. 
We use the convention $H^0( U ) = L^2( U )$ and $\hh^0( U ) = \ll^2( U )$. Sobolev spaces on the boundary $\partial U$ 
are defined accordingly (see, e.g. \cite{ref:mclean2000}), and will be denoted as $W^{s,p}(\partial U)$, $s\in\rrr$, $1<p<\infty$. 



\subsection{An overview}\label{section:overview}
Korn's inequality was first introduced in a pioneering work by Arthur Korn in 1906 \cite{ref:korn1906}. For an open and bounded domain $\Omega$ of $\rrr^d$, $d\geq 2$, A. Korn showed the existence of a positive constant $C>0$ such that
\begin{align}
    \|\nabla\u\|_{0,\Omega} \leq C\|\bepsilon(\u)\|_{0,\Omega},\label{eq:introfirstkorn}
\end{align}
for any vector field $\u$ in $\hh^1(\Omega)$ subject to a zero Dirichlet trace along the boundary of $\Omega$. 
Here $\bepsilon(\u)$ is the strain tensor or the symmetric part of the tensor $\nabla\u$. The inequality in \eqref{eq:introfirstkorn} is usually referred to as {\it Korn's first inequality}. In a second publication \cite{ref:korn1909}, A. Korn proved that the same inequality as in \eqref{eq:introfirstkorn} holds for vector fields $\u$ in $\hh^1(\Omega)$ satisfying the free-rotation condition
\begin{align*}
    \int_\Omega \curl\,\u = \zero,
\end{align*}
with $\curl(\cdot)$ denoting the usual rotation operator. This version is known as {\it Korn's second inequality}. 

Another way to state Korn's inequality is as follows:
\begin{align}
    \|\u\|_{1,\Omega} \leq C\big(\|\bepsilon(\u)\|_{0,\Omega}+\|\u\|_{0,\Omega}\big),\label{eq:introkorn}
\end{align}
for any vector field $\u$ in $\hh^1(\Omega)$. 
This inequality is often simply called {\it Korn's inequality}. One can show that \eqref{eq:introkorn} implies Korn's first inequality in \eqref{eq:introfirstkorn} for functions for which Poincar\'e's inequality holds. Additionally, \eqref{eq:introkorn} can be used to show that Korn's second inequality holds whenever the inclusion $\hh^1(\Omega)\hookrightarrow \ll^2(\Omega)$ is compact.

Note that no extra conditions on the functions defined on $\Omega$ are needed to establish the inequality in \eqref{eq:introkorn}. However, that is not the case for the Korn's first and second inequalities. For instance, \eqref{eq:introfirstkorn} is violated for certain vector fields in $\hh^1(\Omega)$: the so-called {\it rigid motions}, which are vector fields with strain-free energy. The space of rigid motions is defined as
\begin{align}
    \rr\mm(\Omega) :=\big\{\v\in\hh^1(\Omega) :\,\v(\x) = \a + \bb\x,\,\,\a\in\rrr^d,\,\,\bb\in\rrr^{d\times d},\,\,\bb^\transpose=-\bb,\,\,\x\in\Omega\big\}.\label{eq:rididmotionsdef}
\end{align}
Indeed, one can see that $\bepsilon(\cdot)$ defines a linear and bounded operator in $\hh^1(\Omega)$ whose kernel exactly coincides with the space of all rigid motions. From the definition of the Cauchy tensor for linear elastic materials one can see that the kernel of the strain tensor and the Cauchy tensor coincide.

A first proof of Korn's inequality in \eqref{eq:introkorn} for Lipschitz domains was given by J.A. Nitsche \cite{ref:nitsche1981}; this proof is based on the existence of an extension operator in Sobolev spaces. A different proof provided in \cite{ref:kikuchi1988,ref:ting1972} uses the Calder\'on-Zygmund theory of singular integral operators. 




An interesting version of Korn's inequality involving a semi-norm of vector fields in $\hh^1(\Omega)$ satisfying certain conditions is discussed in 
\cite{ref:brenner2003}. In this paper the author considers a bounded semi-norm $\Phi:\hh^1(\Omega)\to\rrr$ so that only pure translations live in its kernel.
In this case a new version of Korn's inequality holds:
\begin{align}
    \|\u\|_{1,\Omega} \leq C_\Phi\big(\|\bepsilon(\u)\|_{0,\Omega}+\Phi(\u)\big)\quad\forall\,\u\in\hh^1(\Omega),\label{eq:introkornseminorm}
\end{align}
where the constant $C_\Phi>0$ is such that $\Phi(\u) \leq C_\Phi\|\u\|_{1,\Omega}$ for any $\u\in\hh^1(\Omega)$.

A similar version of the inequality in \eqref{eq:introkornseminorm} was proven in a recent article \cite{ref:damlamian2018}. 
Specifically, Korn's inequality in \cite[Theorem 3.1]{ref:damlamian2018} reads as follows: 
\begin{quote}
    {\it Let $O$ be a Korn-Wirtinger domain. If $F:H^1(O)\to\rrr$ is a Lipschitz map whose restriction to the subset $\rrrr$ of rigid motions is bounded below by a norm on $\rrrr$. Then there exists a constant $C$ such that
    \begin{align}
     \forall \,u\in H^1(O),\quad \|u\|_{H^1(O)}\leq C\big(F(u) + \|e(u)\|_{L^2(O)}\big).\label{eq:introkornlmap}
    \end{align}
    }
\end{quote}
In the same manuscript, a Korn-Wirtinger domain is defined as follows \cite[Definition  2.2]{ref:damlamian2018}:

\begin{quote}
    {\it A bounded connected open domain $O$ is a Korn-Wirtinger domain if there is a constant $C>0$, depending only on $O$, such that for every $v$ in $H^1(O)$, there is a rigid motion $r(v)$ in $\rrrr$ with
\begin{align*}
    \|v-r(v)\|_{H^1(\Omega)} \leq C \|e(v)\|_{L^2(\Omega)}.
\end{align*}
    }
\end{quote}
Note that any semi-norm is a Lipschitz continuous map. However, \eqref{eq:introkornlmap} is not an immediate generalization of \eqref{eq:introkornseminorm} since $F$ fails to  have only non-zero translations in its kernel. We additionally note that important information about the geometry of the domain
might be lost in the definition of a Korn-Wirtinger domain. For instance, the regularity of the boundary of $O$ does not seem to be readily obtained by simply studying the definition above. 


An extension of the inequality in \eqref{eq:introkorn} to $(\epsilon,\delta)$-domains (also called Jones domains) was given in \cite{ref:duran2004}. 
The notion of $(\epsilon,\delta)$-domains was first introduced by P.W. Jones in 1981 \cite{ref:jones1981}. 
Concretely, a Jones domain is defined as follows. 
\begin{definition}
Let $\Omega\subseteq\rrr^d$ be an open and bounded domain and let $\epsilon,\delta>0$ be given. We say that $\Omega$ is a {\it Jones domain} if for any $\x,\,\y\in\Omega$ such that $\|\x-\y\|<\delta$, there is a rectifiable curve $\gamma:[0,1]\to\Omega$ such that $\gamma(0) = \x$, $\gamma(1) = \y$, and
\begin{align*}
    \ell(\gamma) \leq \frac{\|\x-\y\|}{\epsilon},\quad \epsilon\frac{\|\gamma(t)-\y\|\|\gamma(t)-\y\|}{\|\x-\y\|} \leq \inf_{\w\in\partial\Omega}\|\gamma(t)-\w\|,\,\,\forall\,t\in[0,1],
\end{align*}
where $\ell(\gamma)$ denotes the Euclidean length of $\gamma$.
\end{definition}
 We note that Lipschitz domains and star-shaped domains form a subclass of Jones domains. It was shown in \cite{ref:duran2004} that Korn's inequality in \eqref{eq:introkorn} remains true on Jones domains.
It was additionally shown that this Korn's inequality remains true for vector fields $\u$ in $\ww^{1,p}(\Omega)$, for any $1<p<+\infty$,
%
where the constant $C>0$ depends on $\Omega$, the exponent $p$, $\epsilon$, $\delta$ and the dimension $d\geq 2$. 


Later, it was shown in \cite{ref:duran2006} that Korn's inequality in \eqref{eq:introkorn} holds for John domains and for vector fields in $\ww^{1,p}(\Omega)$, $1<p<+\infty$. Such domains are defined as follows.
\begin{definition}
Let $\Omega$ be an open and bounded domain in $\rrr^d$, $d\geq2$. 
Given $0<\alpha\leq \beta <\infty$, the domain $\Omega$ is a {\it John domain} if there is a point $\x_0\in\Omega$ such that for every point $\x\in\Omega$, $\x\neq\x_0$, there is a rectifiable curve $\gamma:[0,\|\x-\x_0\|]\to\Omega$ such that $\gamma(0) = \x$, $\gamma(\|\x-\x_0\|) = \x_0$ and
\begin{align*}
    \|\x-\x_0\|\leq \beta,\quad \frac{\alpha\cdot t}{\|\x-\x_0\|}\leq\inf_{\z\in\partial\Omega}\|\gamma(t)-\z\|,\,\, \forall\,t\in[0,\|\x-\x_0\|].
\end{align*}
The point $\x_0\in\Omega$ is called the {\it centre} of $\Omega$.
\end{definition}
John domains were first introduced by F. John in 1961 \cite{ref:john1961} and named after him by O. Martio and J. Sarvas in 1978 \cite{ref:martio1978}. The class of John domains consists of very general domains: Lipschitz domains and the Koch snowflake domains belong to this class.
We note that not every John domain is a Jones domain. In fact, the domain $\Omega$ defined as
\begin{align*}
    \Omega := B(\zero,1)\backslash \big\{\x\in\rrr^d:\,x_1 \geq 0,\, x_d = 0\big\}\subseteq \rrr^d,
\end{align*}
is a John domain but not a Jones domain.
Some characterizations and more examples of these domains can be found in
\cite{ref:lopez2018,ref:brewster2014,ref:duran2006,ref:martio1978}.

Finally, we comment that a version of Korn's inequality in \eqref{eq:introkorn} was recently generalized to Banach spaces such as Lebesgue, Lorentz, and Zygmund spaces defined on open sets of $\rrr^d$; see \cite{ref:cianchi2020}.

\subsection{Korn's inequality and continuous mappings}\label{subsection:extension}
Let $\Omega$ be an open and bounded domain $\Omega$ in $\rrr^d$, $d\geq 2$. Let us first recall the following result on compact embedding of Hilbert spaces. 
\begin{theorem}[{\cite[Theorem 3.27]{ref:mclean2000}}]\label{result:inclusion}
 If $U$ is an open and bounded domain and $-\infty < s < t <+\infty$, then $H^t(U)$ is compactly included in $H^s(U)$.
\end{theorem}
It immediately follows from this result that the inclusion $\hh^t(\Omega)\hookrightarrow\hh^s(\Omega)$ is compact for any $s<t$.
We provide a generalized Korn's inequality for vector fields in $\hh^1(\Omega)$. However, the technique of our proof can be easily extended to vector fields in $\ww^{1,p}(\Omega)$, for any $1<p<+\infty$.
The main theorem is as follows.
\begin{theorem}\label{result:kornsh1}
 Assume $\Omega$ is an open and bounded Jones domain or John domain of $\rrr^d$, $d\geq 2$. Let $F:\hh^1(\Omega)\to\rrr$ be a continuous mapping and define the set
  \begin{align*}
    N(F):=\Big\{\v\in\hh^1(\Omega):\,F(\v) = 0\Big\},
\end{align*} Assume further that the mapping $F$ satisfies one of the following conditions
\begin{align}
 N(F)\cap \rr\mm(\Omega) = \{\zero\}\quad\text{or}\quad N(F)\cap \rr\mm(\Omega) = \emptyset,\label{eq:zerocap}
\end{align}
where $\rr\mm(\Omega)$ is the space of rigid motions of $\Omega$ as defined in \eqref{eq:rididmotionsdef}. Then there exists a constant $C>0$, depending only on $\Omega$ and $F$ such that
 \begin{align}
 \|\u\|_{1,\Omega} \leq C\Big(\|\bepsilon(\u)\|_{0,\Omega} + |F(\u)|\Big),\quad\forall\, \u \in 
\hh^1(\Omega).\label{eq:extendedkorns}
\end{align}
\end{theorem}
\begin{proof}
 For a contradiction, suppose there is a sequence $\u_n\in \hh^1(\Omega)$ such that
 \begin{align*}
  \|\u_n\|_{1,\Omega} = 1,\quad \|\bepsilon(\u_n)\|_{0,\Omega} + |F(\u_n)| < \frac{1}{n},\quad\forall\,n\in\nnn.
 \end{align*}
 Since $\{\u_n\}$ is a bounded sequence in the $\hh^1$-norm, we know that there is a subsequence $\{\u_{n_k}\}\subseteq \hh^1(\Omega)$ 
of $\{\u_n\}$ and $\u\in \hh^1(\Omega)$ such that $\u_{n_k}\to \u$ weakly in $\hh^1(\Omega)$. Using 
\autoref{result:inclusion} with $U = \Omega$, $s = 0$ and $t = 1$, we obtain the compactness of the inclusion 
$\hh^1(\Omega)\hookrightarrow \ll^2(\Omega)$. This implies that $\u_{n_k}\to \u$ strongly in $\ll^2(\Omega)$.

On the other hand, we see that $|F(\u_n)| \to 0$ and $\|\bepsilon(\u_n)\|_{0,\Omega}\to 0$. Since $\Omega$ is a Jones domain or John domain, 
we can use Korn's inequality in \eqref{eq:introkorn} to get
\begin{align*}
 \|\u_{n_k}-\u_{n_l}\|_{1,\Omega} \leq C\,\Big(\|\bepsilon(\u_{n_k})-\bepsilon(\u_{n_l})\|_{0,\Omega} + 
\|\u_{n_k}-\u_{n_l}\|_{0,\Omega}\Big).
\end{align*}
Since $\|\bepsilon(\u_n)\|_{0,\Omega}\to 0$ and $\{\u_{n_k}\}$ is strongly convergent in $\ll^2(\Omega)$, the inequality above implies that $\{\u_{n_k}\}$ is a Cauchy sequence in $\hh^1(\Omega)$ and thus $\u_{n_k}\to \u$ 
strongly in $\hh^1(\Omega)$. The continuity of $F$ then gives $|F(\u_{n_k})| \to |F(\u)|$ and therefore $\u\in N(F)$. 
Also, the fact that $\|\bepsilon(\u_n)\|_{0,\Omega}\to 0$ and the continuity of $\bepsilon$ in $\hh^1(\Omega)$ implies 
that $\u\in \rr\mm(\Omega)$ and so $\u$ belongs to $N(F)\cap \rr\mm(\Omega)$. 

If $N(F)\cap \rr\mm(\Omega)$ is the empty set, then we come to an immediate contradiction as we have shown that $\u$ belongs to an empty set.

If, on the other hand, $N(F)\cap \rr\mm(\Omega)=\{\zero\}$, then the limit $\u$ is the zero vector field. However, since $\u_{n_k}\to \u$ and $\|\u_{n_k}\|_{1,\Omega} = 1$ for all $k\in\nnn$, we obtain that $\|\u\|_{1,\Omega} = 1$, which contradicts the fact that $\u = \zero$.
%
\end{proof}
Note that the functional $F$ is meant to be different from the $\ll^2$-norm as this case corresponds to the usual Korn's inequality in \eqref{eq:introkorn} and it is used in the proof of the result above.

We note that no extra conditions on the boundedness of $F$ are needed to establish the result above. One can 
see that a bounded functional would be very useful in many cases as the right hand side of inequality in  \eqref{eq:extendedkorns} might go to infinity. However, in the many applications to the different versions of the Korn's 
inequality, this functional $F$ is also bounded above. In this direction, further assuming that $F$ is a Lipschitz continuous functional in $\hh^1(\Omega)$ with constant $L_F>0$, then we can bound the constant $C>0$ in \eqref{eq:extendedkorns} as follows:
\begin{align*}
    \frac{1}{C} &\leq\, \inf_{\v\in\hh^1(\Omega)} \frac{\|\bepsilon(\v)\|_{0,\Omega} + |F(\v)|}{\|\v\|_{1,\Omega}}\\ &\leq \,(1+L_F) + \inf_{\v\in\hh^1(\Omega)}\frac{|F(\zero)|}{\,\,\,\,\|\v\|_{1,\Omega}}\\
    &=\, 1 + L_F.
\end{align*}
We note that the main steps in the proof of \autoref{result:kornsh1} are the use of the compact embedding $\hh^1(\Omega)\hookrightarrow\ll^2(\Omega)$, and the fact that Korn's inequality in \eqref{eq:introkorn} holds in $\hh^1(\Omega)$. In \cite{ref:duran2004} this version of Korn's inequality was extended to vector fields in $\ww^{1,p}(\Omega)$, whenever $\Omega$ is a Jones (or $(\epsilon,\delta)$) domain and $1<p<+\infty$, whereas authors in \cite{ref:duran2006} extended the same inequality to vector fields in $\ww^{1,p}(\Omega)$ in case $\Omega$ is a John domain and $1<p<+\infty$. Therefore, we can see that the inequality given in \autoref{result:kornsh1} can be further extended to vector fields in the Sobolev space $\ww^{1,p}(\Omega)$, provided $F:\ww^{1,p}(\Omega)\to\rrr$ is continuous, the strain tensor $\bepsilon(\cdot)$ belongs to $\lll^q(\Omega)$, and the inclusion $\ww^{1,p}(\Omega)\hookrightarrow\ll^q(\Omega)$ is compact, for suitable values of $1<p,q<+\infty$.



In the next section we introduce Steklov eigenvalues for the Lam\'e operator and establish the existence of a countable spectrum on Lipschitz domains. Even though the main result of this section was proven for Jones and John domains, Korn's inequality in \autoref{result:kornsh1} is used on Lipschitz domains.

\section{The Steklov-Lam\'e eigenproblem}\label{section:steklovlame}
Let $\Omega$ be an open, bounded and simply connected domain in $\rrr^d$, $d\geq2$, with Lipschitz boundary $\Gamma := \partial\Omega$. 
Assume an isotropic elastic material occupies the region $\Omega$. We denote by $\n$ the outer unit normal vector on $\Gamma$. Let $\u$ be a small displacement of the points in $\Omega$ after some deformation. The stress tensor of the elastic material is
\begin{align*}
 \bsigma(\u) = 2\mu\bepsilon(\u) + \lambda\tr(\bepsilon(\u))\ii\quad\text{in $\Omega$},
\end{align*}
where $\mu$ and $\lambda$ are the usual Lam\'e parameters, satisfying the condition in \eqref{eq:intro-lame-parameters}, namely
\begin{align}
 \mu>0,\quad \lambda + \left(\frac{2}{d}\right)\mu > 0,\label{eq:lame-parameters}
\end{align}
and $\bepsilon(\u):= \frac{1}{2}(\nabla\u+\nabla\u^\transpose)$ is the strain tensor or symmetric part of the deformation 
tensor $\nabla\u$. 
The eigenvalue problem we are interested in reads: find non-zero displacements $\u$ of $\Omega$, 
and frequencies $w\in\ccc$ such that
\begin{subequations}\label{eq:steklovlame}
 \begin{align}
  -\bdiv\,\bsigma(\u) = \zero\quad\text{in $\Omega$},\label{eq:steklovlame1}\\
  \bsigma(\u)\n = w\,p\,\u\quad\text{on $\Gamma$},\label{eq:steklovlame2}
 \end{align}
\end{subequations}
where the parameter $p\in L^\infty(\Gamma)$ 
satisfies
\begin{align}
    p_0 \leq p\quad\text{a.e. on $\Gamma$},\label{eq:lowerboundrho}
\end{align}
for some fixed positive constant $p_0$.
Pairs $(w,\u)$ solving this eigenvalue problem will be called Steklov-Lam\'e eigenpairs. It is easy to see, as for the 
Neumann eigenvalue problem for linear elasticity (usually called the {\it traction free eigenvalue problem} \cite{ref:babuskaosborn1991}), that $w = 0$ is an eigenvalue of this problem with eigenvectors lying 
in $\rr\mm(\Omega)$, the space of all rigid motions defined in \eqref{eq:rididmotionsdef}. 
The divergence theorem implies that $w = 0$ is an Steklov-Lam\'e eigenvalue with eigenvectors belonging to $\rr\mm(\Omega)$ for any Lipschitz domain $\Omega$.
The following result establishes that the Steklov-Lam\'e eigenproblem is self adjoint and has real eigenvalues.
\begin{theorem}\label{result:realevs}
The Steklov-Lam\'e eigenproblem in \eqref{eq:steklovlame} is self-adjoint. In addition, if $(w,\u)$ is an Steklov-Lam\'e eigenpair, then $w\in\rrr$.
\end{theorem}
\begin{proof}
 Let $\v$ be a smooth function. Multiplying \eqref{eq:steklovlame1} by $\v$, integrating by parts and using the boundary condition in \eqref{eq:steklovlame2} gives
 \begin{align*}
     w\,(p\,\u,\v)_{0,\Gamma} = (\bsigma(\u),\bepsilon(\v))_{0,\Omega}.
 \end{align*}
 Now, note that
 \begin{align*}
     (\bsigma(\u),\bepsilon(\v))_{0,\Omega} = 2\mu\,(\bepsilon(\u),\bepsilon(\v))_{0,\Omega} + \lambda\,(\div\,\u,\div\,\v)_{0,\Omega}.
 \end{align*}
Then integrating by parts once more we arrive at
\begin{align*}
     w\,(p\,\u,\v)_{0,\Gamma} = -(\u,\bdiv\,\bsigma(\v))_{0,\Omega} + (\u,\bsigma(\v)\n)_{0,\Gamma}.
 \end{align*}
 Since $(p\,\u,\v)_{0,\Gamma} = (\u,p\,\v)_{0,\Gamma}$, the identity above shows that the Steklov-Lam\'e eigenproblem in \eqref{eq:steklovlame} is self-adjoint.
 
 On the other hand, we see from the above that any eigenvalue $w$ satisfies
 \begin{align*}
     w = \frac{2\mu\,(\bepsilon(\u),\bepsilon(\v))_{0,\Omega} + \lambda\,(\div\,\u,\div\,\v)_{0,\Omega}}{(p\,\u,\v)_{0,\Gamma}},
 \end{align*}
 for all $\u$ and $\v$ which do not vanish on $\Gamma$. Since the inner products $(\cdot,\cdot)_{0,\Omega}$ and $(\cdot,\cdot)_{0,\Gamma}$, and the Lam\'e parameters $\lambda$ and $\mu$ are real, we obtain from identity above that the eigenvalues $w$ are also real.
\end{proof}

We propose the use of a weak formulation to study both the continuous and the discrete spectra of this problem. However, the zero eigenvalue mentioned just before implies that, as suggested in  \cite{ref:babuskaosborn1991}, a shift needs to be added to the formulation to obtain a coercive bilinear form.
We employ the following weak formulation: find Steklov-Lam\'e eigenpairs $\u\in\hh^1(\Omega)$ and $w\in\rrr$, $\u\neq\zero$, such that
\begin{align}
    a(\u,\v) = \kappa\, b(\u,\v)\quad\forall\,\v\in\hh^1(\Omega),\label{eq:weakform}
\end{align}
where $\kappa := w+1$ and the bilinear forms $a$ and $b$ are defined as
\begin{align*}
    a(\u,\v) &:= (\bsigma(\u),\bepsilon(\v))_{0,\Omega} + (p\,\u,\v)_{0,\Gamma}\quad\forall\,\u,\v\in\hh^1(\Omega),\\
    b(\u,\v) &:= (p\,\u,\v)_{0,\Gamma}\quad\forall\,\u,\v\in\hh^1(\Omega).
\end{align*}
We will also be using the induced operators of these bilinear forms, $A,B:\hh^1(\Omega)\to \hh^1(\Omega)'$, defined as
\begin{align}
    [A(\u),\v]_{\hh^1(\Omega)} = a(\u,\v),\quad [B(\u),\v]_{\hh^1(\Omega)} = b(\u,\v),\quad\forall\,\u,\v\in\hh^1(\Omega).
\end{align}
Our aim is to characterize the Steklov-Lam\'e eigenpairs through the use of \eqref{eq:weakform}. To this end, we need to establish some important properties of the corresponding solution operator. As we will see, this will be achieved with the use of Korn's inequality as given in \autoref{result:kornsh1}, the properties of the trace operator along the boundary, and the application of the well known Spectral theorem.

In the next section we provide a spectral characterization of the continuous problem given by \eqref{eq:weakform}.




\subsection{Spectral characterization}\label{section:spectrum}
%
%
Let us first recall some properties of the trace operator on Lipschitz domains. The next result concerns the existence of the trace operator for smooth vector fields.
\begin{lemma}[{\cite[Lemme 1.3-5]{ref:raviart1983}}]\label{result:definitiontrace}
 Let $U$ be a bounded Lipschitz domain in $\rrr^d$, $d\geq 2$ and let $C^\infty_0(\overline{U})$ be the set of all infinitely differentiable functions on $\overline{U}$ with support in an open set $O$ such that $\overline{U}\subseteq O$. Let $\gamma_0:C^\infty_0(\overline{U})\to L^2(\partial U)$ be the mapping defined by
 \begin{align}
     \gamma_0(v) = v|_{\partial U},\quad\forall\,v\in C^\infty_0(\overline{U}).\label{eq:definitiontrace}
 \end{align}
 Then there exists a constant $c>0$ such that
\begin{align*}
    \|\gamma_0(v)\|_{0,\partial U}\leq\, c\|v\|_{1,U},\quad\forall\,v\in C^\infty_0(\overline{U}).
\end{align*}
\end{lemma}
The following result establishes the extension of the trace operator as defined in \eqref{eq:definitiontrace}.
\begin{theorem}[{see e.g. \cite[Theorem 1.5]{ref:gatica2014}}]\label{result:extensiontrace}
Let $U$ be a bounded Lipschitz domain in $\rrr^d$, $d\geq 2$. Then the mapping $\gamma_0:C^\infty_0(\overline{U})\to L^2(\partial U)$ can be extended by continuity and density to a linear and bounded operator $\gamma_0:H^1(U)\to L^2(\partial U)$ such that it satisfies \eqref{eq:definitiontrace}.
\end{theorem}
The trace operator for vector fields in $\hh^1(\Omega)$, also denoted by $\gamma_0$, is then defined componentwise by
\begin{align}
    \gamma_0(\v) := (\gamma_0(v_1),\ldots,\gamma_0(v_d)),\quad \forall\,\v\in\hh^1(\Omega).\label{eq:vectortrace}
\end{align}
Thus, by \autoref{result:extensiontrace} the operator $\gamma_0:\hh^1(\Omega)\to\ll^2(\Gamma)$ is a linear and bounded mapping with continuity constant $c>0$.

We note that the continuity of the trace operator $\gamma_0$ in $\hh^1(\Omega)$ implies that the bilinear forms $a$ and $b$ are continuous bilinear forms, with
\begin{subequations}\label{eq:boundsforab}
 \begin{align}
    |\,a(\u,\v)|\leq&\, \max\big\{\lambda+2\mu,c^2\|p\|_{\infty,\Gamma}\big\} \|\u\|_{1,\Omega}\|\v\|_{1,\Omega},\label{eq:boundfora}\\
    |\,b(\u,\v)|\leq&\, \big(c^2\|p\|_{\infty,\Gamma}\big) \|\u\|_{1,\Omega}\|\v\|_{1,\Omega},\label{eq:boundforb}
\end{align}
\end{subequations}
where $c>0$ is the continuity constant of the trace operator.
From the definition of the bilinear form $b$ we see that $b(\u,\u) \geq 0$ for all $\u\in\hh^1(\Omega)$, and $b(\u,\v) = 0$ whenever $\u$ or $\v$ vanish along the boundary $\Gamma$. This directly implies that the kernel of the induced operator $B$ is
\begin{align*}
    N(B) = \hh^1_0(\Omega).
\end{align*}
On the other hand, the Rayleigh quotient then gives
\begin{align}
 \kappa = \frac{a(\u,\u)}{b(\u,\u)},\quad\forall\,\u\in\hh^1(\Omega),\,\, b(\u,\u)>0.\label{eq:steklovrayleigh}
\end{align}
We then have that, thanks to \autoref{result:realevs} and the Rayleigh quotient above, all possible eigenvalues $\kappa$ of \eqref{eq:weakform} are non-negative provided the associated eigenvectors $\u\in\hh^1(\Omega)$ satisfy $b(\u,\u)>0$. 

We see that 
for any $\u\in\hh^1_0(\Omega)$, $b(\u,\v) = 0$ for all $\v\in\hh^1(\Omega)$. If such $\u\in\hh^1_0(\Omega)$ also satisfies $a(\u,\v) = 0$ for all $\v\in\hh^1(\Omega)$, then $\kappa\in\rrr$ is not an eigenvalue but it belongs to the spectrum of the eigenproblem. Thus, we need to show, using the Korn's inequality in $\hh^1_0(\Omega)$, that the kernels of the induced operators $A$ and $B$ do not share any non-zero elements.
To see this, let us consider the following subset of $\hh^1(\Omega)$:
\begin{align}
    \hh := \Big\{\u\in\hh^1(\Omega):\, \big(\bsigma(\u),\bepsilon(\v)\big)_{0,\Omega} = 0,\quad\forall\,\v\in\hh^1_0(\Omega)\Big\}.\label{eq:hspace}
\end{align}
We note that the space $\hh$ is a closed subspace of $\hh^1(\Omega)$ with the usual norm $\|\cdot\|_{1,\Omega}$. Next, we show that $\hh\cap\hh^1_0(\Omega) = \{\zero\}$. Indeed, if $\u\in\hh\cap\hh^1_0(\Omega)$, then the definition of $\hh$ implies that 
%
\begin{align*}
    \big(\bsigma(\u),\bepsilon(\v)\big)_{0,\Omega} = 0,\quad\forall\,\v\in\hh^1_0(\Omega).
\end{align*}
Taking $\v:=\u$ in the expression above and after some algebraic manipulations we arrive at
\begin{align*}
    0 = \big(\bsigma(\u),\bepsilon(\u)\big)_{0,\Omega} \geq \min\left\{2\mu,d\left(\lambda + \left(\frac{2}{d}\right)\mu\right)\right\}\|\bepsilon(\u)\|_0^2.
\end{align*}
Korn's first inequality in \eqref{eq:introfirstkorn} for vector fields in $\hh^1_0(\Omega)$ shows that $\|\bepsilon(\u)\|_{0,\Omega}$ and $\|\u\|_{1,\Omega}$ are equivalent norms and so $\u = \zero$, showing that $\hh$ and $\hh^1_0(\Omega)$ only share the zero vector. 

On the other hand, we have that $\hh^1(\Omega) = \hh + \hh^1_0(\Omega)$. To see this, let us start by picking an element $\u$ in $\hh^1(\Omega)$. Consider the problem of finding $\z\in \hh^1_0(\Omega)$ such that 
\begin{align}
    \big(\bsigma(\z),\bepsilon(\v)\big)_{0,\Omega} = \big(\bsigma(\u),\bepsilon(\v)\big)_{0,\Omega}\quad\forall\,\v\in\hh^1_0(\Omega).\label{eq:auxiliaryproblem}
\end{align}
Notice that $\big(\bsigma(\u),\bepsilon(\cdot)\big)_{0,\Omega}$ defines a bounded and linear operator in $\hh^1_0(\Omega)$. In addition, the bilinear form $\big(\bsigma(\cdot),\bepsilon(\cdot)\big)_{0,\Omega}$ is coercive in $\hh^1_0(\Omega)$. Thus, the Lax-Milgram lemma implies that there is a unique solution $\z\in\hh^1_0(\Omega)$ of the problem in \eqref{eq:auxiliaryproblem}.
Now define $\w:= \u - \z$. Thanks to the problem in \eqref{eq:auxiliaryproblem} we see that
\begin{align*}
    \int_\Omega \bsigma(\w):\bepsilon(\v) = 0\quad\forall\,\v\in\hh^1_0(\Omega).
\end{align*}
We have that $\w\in\hh$ and $\z\in\hh^1_0(\Omega)$ are such that $\u = \w + \z$, proving that $\hh^1(\Omega)$ can be decomposed as the sum of $\hh$ and $\hh^1_0(\Omega)$. Altogether, we have the following result.
\begin{theorem}\label{result:decompostion}
With $\hh$ defined as in \eqref{eq:hspace}, the space $\hh^1(\Omega)$ can be decomposed as follows:
\begin{align}
    \hh^1(\Omega) = \hh \oplus \hh^1_0(\Omega).\label{eq:spectraldecomposition}
\end{align}
\end{theorem}
These results together with the Rayleigh quotient
imply that all eigenvalues are real, that eigenfunctions corresponding to non-negative eigenvalues lie in $\hh$, and $\hh^1_0(\Omega)$ constitutes a generalized eigenspace of the corresponding solution operator, to be defined in the next section. 

In the following section we prove the existence of a countable Steklov-Lam\'e spectrum.

\subsection{Existence of a countable spectrum}
The proof of the existence of a countable spectrum for the Steklov-Lam\'e eigenproblem turns out to depend on Korn's inequality in the form given in \eqref{eq:extendedkorns} and the properties of the trace operator. Korn's inequality is used to show that the bilinear form $a$ is $\hh^1$-elliptic, while the properties of the trace operator are used to show the compactness of the corresponding solution operator.

Let us first recall the following result on the compactness of the trace operator as defined in \eqref{eq:definitiontrace} and \autoref{result:extensiontrace}. The proof for this result can be found in \cite[Theorem 3.81]{ref:demengel2012} in the case of smooth domains, and in \cite[Theorem 6.2]{ref:necas2011} in the case of Lipschitz domains. An extension to {\it $W^{1,p}$-extension domains} was given in \cite{ref:biegert2009}.
\begin{theorem}\label{result:compactnesstrace}
 Let $U$ be a bounded Lipschitz domain in $\rrr^d$, $d\geq 2$. Then the mapping $\gamma_0:H^1(U)\to L^{q}(\partial U)$ is compact if $d = 2$ and $1\leq q<+\infty$, or if $d>3$ and $1\leq q < \frac{2(d-1)}{(d-2)}$.
\end{theorem}
In particular, this result implies that the trace operator $\gamma_0:\hh^1(\Omega)\to\ll^2(\Gamma)$ as defined in \eqref{eq:vectortrace} is compact.

As a first attempt let us now define a solution operator for the weak form in \eqref{eq:weakform}. We define this solution operator $T:\ll^2(\Gamma)\to\hh^1(\Omega)$ as $T(\f) = \u$, where $\u\in\hh^1(\Omega)$ and $\f\in\ll^2(\Gamma)$ satisfy the source problem
\begin{align}
 a(\u,\v) = b(\f,\v),\quad\forall\,\v\in\hh^1(\Omega).\label{eq:soloperator}
\end{align}
Note that the boundedness of the bilinear form $b$ as used in \eqref{eq:soloperator} is slightly different from that in \eqref{eq:boundforb}. In fact, since we only have $\f\in\ll^2(\Gamma)$, we have the following bound for $b$
\begin{align}
    |b(\f,\v)|\leq \big(c\|p\|_{\infty,\Gamma}\big)\|\f\|_{0,\Gamma}\|\v\|_{1,\Omega},\quad\forall\,\f\in\ll^2(\Gamma),\,\,\forall\,\v\in\hh^1(\Omega).\label{eq:2ndboundforb}
\end{align}
We take advantage of the continuity of the trace operator $\gamma_0:\hh^1(\Omega)\to\ll^2(\Gamma)$ to show that $a(\cdot,\cdot)$ is $\hh^1$-elliptic. The following result is a direct application of \autoref{result:kornsh1}.
\begin{theorem}\label{theorem:ellipticity}
Let $\Omega$ be a Lipschitz continuous domain in $\rrr^d$, $d\geq 2$ with boundary $\Gamma:=\partial\Omega$. Also let $\gamma_0:\hh^1(\Omega)\to\ll^2(\Gamma)$ be the trace map. Then there is a constant $C>0$ such that 
 \begin{align}
  \|\u\|_{1,\Omega} \leq C\Big(\|\bepsilon(\u)\|_{0,\Omega} + 
\|\gamma_0(\u)\|_{0,\Gamma}\Big),\quad\forall\,\u\in\hh^1(\Omega).\label{eq:kornsineq}
 \end{align}
\end{theorem}
\begin{proof}
 The proof follows by using \autoref{result:kornsh1} with the functional $F(\u) := 
\|\gamma_0(\u)\|_{0,\Gamma}$. In fact, Since the trace operator is linear and bounded, the functional $F$ is Lipschitz continuous. Also, we immediately see that $\hh^1_0(\Omega) = N(F)$. For a vector field $\v\in N(F)\cap\rr\mm(\Omega) = \hh^1_0(\Omega)\cap\rr\mm(\Omega)$, we can use Korn's inequality in $\hh^1_0(\Omega)$ once again to show that $\v=\zero$. Thus, $N(F)\cap\rr\mm(\Omega) = \{\zero\}$. Therefore, the result follows from a straightforward 
application of \autoref{result:kornsh1}.
\end{proof}
Now, using the inequality in \eqref{eq:kornsineq}, for any $\u\in\hh^1(\Omega)$, we can get
\begin{align*}
 a(\u,\u) &\geq \min\left\{2\mu,d\left(\lambda + \left(\frac{2}{d}\right)\mu\right)\right\}\|\bepsilon(\u)\|_0^2 + 
p_0\|\u\|_{0,\Gamma}^2\\
	  &\geq \frac{1}{2C^2}\min\left\{p_0,2\mu,d\left(\lambda + \left(\frac{2}{d}\right)\mu\right)\right\}\|\u\|_{1,\Omega}^2
\end{align*}
Thus, the bilinear form is $\hh^1$-elliptic, i.e. $\alpha \|\u\|_{1,\Omega}^2 \leq\, a(\u,\u)$, for all $\u\in\hh^1(\Omega)$, with $\alpha>0$ defined as
\begin{align}
\alpha := \frac{1}{2C^2}\min\left\{p_0,2\mu,d\left(\lambda +\left(\frac{2}{d}\right)\mu\right)\right\}.\label{eq:ellipticconstant}
\end{align}
Note that this also implies that the bilinear form $a$ defines an inner product, equivalent to the usual inner product of $\hh^1(\Omega)$. 

The Lax-Milgram lemma then gives the unique solvability of the weak problem in \eqref{eq:soloperator}. Thus, the solution operator is well-defined as a linear operator. In addition, the continuity of $b$ as shown in \eqref{eq:2ndboundforb} together with the coercivity of $a$ in $\hh^1(\Omega)$ imply that $T$ is bounded, with
\begin{align}
 \|T(\f)\|_{1,\Omega} \leq\, \left(\frac{c\|p\|_{\infty,\Gamma}}{\alpha}\right) \|\f\|_{0,\Gamma},\quad\forall\,\f\in\ll^2(\Gamma).\label{eq:1stboundforsoloperator}
\end{align}
We notice that, however, the operator $T$ lacks symmetry since the space $\hh^1(\Omega)$ is not included in $\ll^2(\Gamma)$. Nonetheless, we can achieve the symmetry of a different operator related to the solution operator $T$. Recall that the Rayleigh quotient in \eqref{eq:steklovrayleigh} shows that all possible Steklov-Lam\'e eigenvalues need to be non-negative and real numbers. If the solution operator is not self-adjoint then we may end up adding complex eigenvalues that are not part of the spectrum. This says that $T$ is not the most suitable operator to show the existence the Steklov-Lam\'e spectrum, and we need to modify it to be able to characterize the correct eigenpairs of this problem. This shows that even though $T$ appears to be the obvious choice for a solution operator of the Steklov-Lam\'e eigenproblem, it is not the most suitable operator to show the existence of a countable spectrum. We need instead a modification of this operator to be able to characterize the correct spectrum.

Recalling the boundedness of the bilinear form $b$ in \eqref{eq:2ndboundforb},  and the fact that $\hh^{1/2}(\Gamma)\subseteq\ll^2(\Gamma)$, we also have that
\begin{align*}
 \|(T\circ\gamma_0)(\f)\|_{1,\Omega} \leq\, \left(\frac{c^2\|p\|_{\infty,\Gamma}}{\alpha}\right) \|\f\|_{1,\Omega},\quad\forall\,\f\in\hh^1(\Omega),
\end{align*}
where the constant $c>0$ is the continuity of $\gamma_0$ in $\hh^1(\Omega)$ (cf. \eqref{eq:definitiontrace}). Therefore, the composition $T\circ\gamma_0:\hh^1(\Omega)\to\hh^1(\Omega)$ is linear and bounded. In addition, the symmetry of the bilinear forms $a$ and $b$, together with the definition of the problem in \eqref{eq:soloperator} imply that $T\circ\gamma_{0}$ is a self-adjoint operator in the $a(\,\cdot\,,\,\cdot\,)$-inner product. 

We are only left to prove that $T\circ\gamma_0$ is compact. To see this, we recall from \autoref{result:compactnesstrace} that the trace operator is compact from $\hh^1(\Omega)$ to $\ll^2(\Gamma)$. Since $T$ is a bounded and linear operator, \autoref{result:compactnesstrace} implies that $T\circ\gamma_0$ is also compact.

Now, we can easily see that $T\circ\gamma_0(\u) = \nu \u$, $\nu\neq0$, is a solution of \eqref{eq:soloperator} if and only if $\u$ is an eigenvector of \eqref{eq:steklovlame} with eigenvalue $\kappa = \frac{1}{\nu}$. 
Altogether, we conclude, thanks to the Spectral theorem for linear, bounded, self-adjoint and compact operators, that there is a sequence of eigenpairs $(w_n,\u_n)\in\rrr\times\hh^1(\Omega)$, $w_n\to +\infty$ as $n\to+\infty$, such that
\begin{align*}
    (T\circ\gamma_0)(\u_n) = \left(\frac{1}{w_n+1}\right)\u_n,\,\,\forall\,n\in\nnn.
\end{align*}
Note that if $\nu = 0$ then the above properties cannot be guaranteed. However, we see that $\nu = 0$ is contained in the spectrum of $T\circ\gamma_0$ with $N(T\circ\gamma_0)$ as its associated generalized eigenspace (of finite or infinite dimension). We summarize these properties in the following result.

\begin{theorem}
 The solution operator $T\circ\gamma_0:\hh^1(\Omega)\to\hh^1(\Omega)$ is linear, bounded, compact and self-adjoint in the $a(\,\cdot\,,\,\cdot\,)$-inner product. Its spectrum $\sigma(T\circ\gamma_0) = \{0,1\}\cup\{\nu_n:n\in\nnn\}$ is decomposed as follows
 \begin{enumerate}\label{result:steklovlame}
     \item $\nu = 1$ is an eigenvalue of $T\circ\gamma_0|_{\hh}$ with eigenspace $\rr\mm(\Omega)$;
     \item $\nu_n\in(0,1)$ is an eigenvalue of $T\circ\gamma_0|_{\hh}$ with eigenvectors lying in $\hh$;
     \item the accumulation point $\nu = 0$ belongs to $\sigma(T\circ\gamma_0)$ with $\hh^1_0(\Omega)$ as its generalized eigenspace.
 \end{enumerate}
 In addition, eigenfunctions corresponding to different eigenvalues are orthogonal with respect to $b(\cdot,\cdot)$.
\end{theorem}
\begin{proof}
 (1) and (2) of the theorem are direct applications of the Spectral theorem and \autoref{result:decompostion}. For (3), let us recall that $0$ is an eigenvalue of a linear and bounded operator if and only if the corresponding eigenspace is a subset of the kernel of the operator. Then, from \eqref{eq:1stboundforsoloperator} we 
 see that $(T\circ\gamma_0)(\w) = \zero$ as long as $\w$ belongs to $\hh^1_0(\Omega)$. Thus, $\nu = 0$ is an eigenvalue of $T\circ\gamma_0$ with eigenspace $\hh^1_0(\Omega)$.
\end{proof}

This result gives us the following spectral characterization for the spectrum of the Steklov-Lam\'e eigenproblem in \eqref{eq:steklovlame}.
\begin{theorem}\label{result:steklovspectrum}
 The spectrum of the Steklov-Lam\'e eigenproblem in \eqref{eq:steklovlame} is $\{w_n:n\in\nnn\}\cup\{0\}$ and it is decomposed as follows
 \begin{enumerate}
     \item $w = 0$ is an eigenvalue of \eqref{eq:steklovlame} with associated eigenfunctions lying in $\rr\mm(\Omega)$;
     \item $w_n\in(0,+\infty)$ is an eigenvalue of \eqref{eq:steklovlame} with associated eigenfunctions belonging to $\hh$.
 \end{enumerate}
\end{theorem}

In the theory of elasticity one usually finds a more general Robin boundary condition than the one considered in \eqref{eq:steklovlame2}. Let us assume that $\mm\in\lll^\infty(\Gamma)$ is a given symmetric matrix.
Let us further assume that there exists a constant $m>0$ such that the following lower bound holds
\begin{align}
    (\mm\u,\u)_{0,\Gamma}\geq m\|\u\|_{0,\Gamma}^2,\,\,\forall\,\u\in\ll^2(\Gamma).\label{eq:positivedefinite}
\end{align}
Let us consider the following eigenproblem: find eigenpairs $\u\in\hh^1(\Omega)$ and $w\in\ccc$ such that
\begin{align}
  -\bdiv\,\bsigma(\u) = \zero\quad\text{in $\Omega$},\quad
  \bsigma(\u)\n = w\mm\u\quad\text{on $\Gamma$}.\label{eq:gen-steklovlame}
 \end{align}
This form of boundary condition is considered in the study of elasticity since it allows constraints on specific directions of the displacement represented by $\u$, and is sometimes referred to as a {\it gliding} boundary condition.

We see that a similar weak formulation can be obtained for the eigenproblem in \eqref{eq:gen-steklovlame}; we simply replace the inner product $(p\,\u,\v)_{0,\Gamma}$ by $(\mm\,\u,\v)_{0,\Gamma}$. The main difference comes in the upper bounds for the continuity of the bilinear forms (equivalently their induced operators) and the ellipticity constant $\alpha$. For the problem in \eqref{eq:gen-steklovlame}, these bounds are
\begin{align*}
    |\,a(\u,\v)|\leq&\, \big(c^2\max\big\{\lambda+2\mu,\|\mm\|_{\infty,\Gamma}\big\}\big) \|\u\|_{1,\Omega}\|\v\|_{1,\Omega},\\
    |\,b(\u,\v)|\leq&\, \big(c^2\|\mm\|_{\infty,\Gamma}\big) \|\u\|_{1,\Omega}\|\v\|_{1,\Omega},
\end{align*}
and $\alpha =\frac{1}{2C^2}\min\left\{m,2\mu,d\left(\lambda +\left(\frac{2}{d}\right)\mu\right)\right\}$. Thus, since all the necessary conditions for the corresponding solution operator to be well defined, linear, bounded and compact are in place, the results given in \autoref{result:steklovlame} and \autoref{result:steklovspectrum} are true for this case as well.

In the forthcoming section we discuss a conforming finite element scheme to approximate Steklov-Lam\'e eigenpairs of \eqref{eq:steklovlame}. We also follow the theory developed in \cite{ref:babuskaosborn1991} to provide the spectral approximation of this scheme.

\section{Discrete formulation and numerical results}\label{section:discrete}
Let $\Omega$ be a bounded Lipschitz polygonal domain in $\rrr^d$, $d\geq2$. 
Let $\tttt_h$ be a triangulation (by triangles in 2D or tetrahedra in 3D) of $\overline{\Omega}$, with mesh size $h>0$, that is 
$h:=\max\{h_T:T\in\tttt_h\}$, and $h_T$ denoting the diameter of the triangle $T$ in the triangulation. Let $\hh_h$ be a finite dimensional subspace of $\hh^1(\Omega)$.

Let us consider the following discrete formulation of \eqref{eq:weakform}: find $\u_h\in\hh_h$ and $w_h\in\rrr$ such that
\begin{align}
 a(\u_h,\v_h) = \kappa_h\, b(\u_h,\v_h),\,\,\forall\,\v_h\in\hh_h,\label{eq:discrete-form}
\end{align}
where $\kappa_h := w_h+1$, and the bilinear forms $a$ and $b$ are defined as in \ref{section:steklovlame}. In this case, the fact that $\hh_h\subseteq\hh^1(\Omega)$ implies that the bilinear forms $a$ and $b$ are bounded in $\hh_h$. In addition, as shown in the previous section, we can conclude that the bilinear form $a$ is also elliptic on $\hh_h$ with the same ellipticity constant $\alpha>0$ (cf. \eqref{eq:ellipticconstant}). As for the solution operator $T\circ\gamma_0$, we define the solution operator $T_h\circ\gamma_0:\hh^1(\Omega)\to\hh^1(\Omega)$ as $(T_h\circ\gamma_0)(\f) = \u_h$, where $\u_h\in\hh_h$ and $\f\in\hh^1(\Omega)$ satisfy the source problem
\begin{align}
    a(\u_h,\v_h) = b(\f,\v_h),\,\,\forall\,\v_h\in\hh_h.\label{eq:discretesource}
\end{align}
Then the $\hh_h$-ellipticity of $a$ provides the uniqueness of solution in $\hh_h$ of the problem above. This also implies that the discrete solution operator $T_h\circ\gamma_0$ in well-defined as a linear and bounded operator.  
The compactness of $T_h\circ\gamma_0$ is guaranteed since its range is contained in the finite dimensional subspace $\hh_h$. As for the solution operator defined in \eqref{eq:soloperator}, $T_h\circ\gamma_0$ is also self-adjoint.

We see that $(T_h\circ\gamma_0)(\u_h) = \nu_h\u_h$, $\nu_h\neq0$, is a solution of the source problem in \eqref{eq:discretesource} if and only if $\u_h\in\hh_h$ solves \eqref{eq:discrete-form} with eigenvalue satisfying $ \kappa_h = \frac{1}{\nu_h}$. We summarize these properties in the next result.
\begin{theorem}
The spectrum of $T_h\circ\gamma_0$, $\sigma(T_h\circ\gamma_0)$ is decomposed as $\{0\}\cup\{\nu_{h,n}:\,n = 1,\ldots,N_h\}\cup\{1\}$, with $N_h:=\dim(\hh\cap\hh_h)$. In this case we have that
\begin{enumerate}
    \item $\nu_h = 1$ is an eigenvalue of $T_h\circ\gamma_0|_{\hh_h}$ with $\rr\mm(\Omega)\cap\hh_h$ as the associated eigenspace;
    \item $\nu_{h,n}\in(0,1)$, $n=1,\ldots,N_h$ are eigenvalues of $T_h\circ\gamma_0|_{\hh_h}$ with eigenfunctions lying in $\hh\cap\hh_h$;
    \item $\nu_h = 0$ belongs to the spectrum of $T_h\circ\gamma_0$ with $\hh^1_0(\Omega)\,\cap\,\hh_h$ as its associated generalized eigenspace.
\end{enumerate}
In addition, eigenfunctions corresponding to different eigenvalues are orthogonal with respect to $b(\cdot,\cdot)$.
\end{theorem}
\begin{proof}
 (1) and (2) follow directly from the $\hh^1$-ellipticity of the bilinear form $a$. For (3) we see that $\nu_h = 0$ is an eigenvalue of $T_h$ if and only if its associated eigenfunction $\u_h\in\hh_h$ satisfies
 \begin{align}
     \int_{\Gamma} p\,\u_h\cdot\v_h = 0,\quad\forall\,\v_h\in\hh_h.
 \end{align}
Because $p\in L^\infty(\Gamma)$ is bounded below by $p_0$ (cf. \eqref{eq:lowerboundrho}), we obtain that $\gamma_0(\u_h) = \zero$ on $\Gamma$. Then, for $\nu_h = 0$ to be an eigenvalue of $T_h$ one needs to require that the eigenfunctions lie in $\hh^1_0(\Omega)\cap\hh_h$.
\end{proof}
In the subsequent section we establish the spectral approximation properties of the discrete scheme in \eqref{eq:discrete-form}.

\subsection{Spectral approximation}\label{section:approx}
We follow the approach presented in \cite{ref:babuskaosborn1991} to show that the proposed conforming Galerkin scheme in \eqref{eq:discrete-form} provides convergent approximation to the true eigenvalues of \eqref{eq:weakform}, and does not add spurious eigenmodes to the spectrum of $T_h\circ\gamma_0$. Hereafter we assume that for every $\u\in\hh^1(\Omega)$, the following approximation property regarding the space $\hh_h$ holds
\begin{align}
    \lim_{h\to0}\inf_{\v_h\in\hh_h}\|\u-\v_h\|_{1,\Omega} = 0.\label{eq:approx}
\end{align}
The first property we need to check is the convergence of $T_h\circ\gamma_0$ to $T\circ\gamma_0$ in operator norm. We note that the definition of these operators, together with the fact that $a$ is coercive in $\hh^1(\Omega)$, implies that $T_h\circ\gamma_0 = P_h\circ T\circ\gamma_0$, where $P_h:\hh^1(\Omega)\to \hh_h$ is the orthogonal projection with respect to the inner product $a(\cdot,\cdot)$. The approximation property in \eqref{eq:approx} implies that $P_h$ converges to the identity mapping in operator norm. Then, the compactness of $T\circ\gamma_0$, the continuity of $P_h$ and its pointwise convergence to the identity mapping in $\hh^1(\Omega)$ imply that $T_h\circ\gamma_0$ converges to $T\circ\gamma_0$ in norm, that is $\|T_h\circ\gamma_0-T\circ\gamma_0\|\to 0$ as $h\to 0$, where
\begin{align*}
    \|T_h\circ\gamma_0-T\circ\gamma_0\|:= \sup_{\v\in\hh^1(\Omega)}\frac{\|T_h\circ\gamma_0(\v)-T\circ\gamma_0(\v)\|_{1,\Omega}}{\|\v\|_{1,\Omega}}.
\end{align*}
Norm of the operator $T_h\circ\gamma_0$ implies that any conforming scheme of the form given in \eqref{eq:discrete-form} does not add any spurious eigenvalues, see \cite{ref:babuskaosborn1991}, \cite[Section 5.]{ref:meddahi2013} or \cite[Section 1. and Section 2.]{ref:descloux1978-1} for a more detailed discussion.

If $(\kappa,\u)$ and $(\tilde\kappa,\tilde\u)$ are two distinct eigenpairs of \eqref{eq:weakform}, the following identity holds (see \cite[Lemma 9.1]{ref:babuskaosborn1991})
\begin{align*}
    (\tilde\kappa-\kappa)\,b(\tilde\u,\tilde\u) = a(\u-\tilde\u,\u-\tilde\u) - \kappa\,b(\u-\tilde\u,\u-\tilde\u).
\end{align*}
Then since the the discrete formulation and the finite element space $\hh_h$ are conforming, for eigenpairs $(\kappa,\u)$ and $(\kappa_h,\u_h)$ of \eqref{eq:weakform} and \eqref{eq:discrete-form} respectively and such that $b(\u_h,\u_h) = 1$, the continuity of the bilinear forms $a$ and $b$ give the following bound for the eigenvalues $\kappa$ and $\kappa_h$
\begin{align*}
    |\kappa-\kappa_h|\leq \big(\max\{\lambda+2\mu,c^2\|p\|_{\infty,\Gamma}\} + \kappa c^2\|p\|_{\infty,\Gamma}\big)\,\|\u-\u_h\|_{1,\Omega}^2.
\end{align*}
On the other hand, the error $\|\u-\u_h\|_{1,\Omega}$ can be bounded as follows 
\begin{align}
    \|\u-\u_h\|_{1,\Omega}\leq \left(\frac{C}{\alpha}\right)\,\delta(\ee(\kappa),\hh_h),\label{eq:ewapprox}
\end{align}
for some constant $C>0$, independent of $h$. Here $\ee(\kappa)$ is the eigenspace corresponding to the eigenvalue $\kappa$, and $\delta(\cdot,\cdot)$ is the gap between two spaces, and is defined as
\begin{align}
    \delta(\ee(\kappa),\hh_h) := \sup_{\substack{\u\in\ee(\kappa)\\\|\u\|_{1,\Omega}=1}}\inf_{\v_h\in\hh_h}\|\u-\v_h\|_{1,\Omega}.\label{eq:gapdefinition}
\end{align}
Putting back the estimate for $\|\u-\u_h\|_{1,\Omega}$ into the estimate for $|\kappa-\kappa_h|$, we obtain the following estimate for the eigenvalues:
\begin{align}
    |\kappa-\kappa_h|\leq\left(\frac{C}{\alpha}\right)^2\big(\max\{\lambda+2\mu,c^2\|p\|_{\infty,\Gamma}\} + c^2(\kappa+1)\|p\|_{\infty,\Gamma}\big)\delta(\ee(\kappa),\hh_h)^2.\label{eq:evapprox}
\end{align}
The characterization given by the Spectral theorem of the point spectrum of the Steklov-Lam\'e eigenproblem implies that for a Steklov-Lam\'e eigenvalue $\kappa>0$, its associated eigenspace $\ee(\kappa)$ is of finite dimension. Thus, the assumption in \eqref{eq:approx} implies that $\delta(\ee(\kappa),\hh_h)\to0$ as $h\to0$. Therefore, \eqref{eq:ewapprox} provides the convergence of $\u_h$ to the eigenfunction $\u$ in the $\hh^1$-norm, as well as the convergence of $\kappa_h$ to  the eigenvalue $\kappa$ as a consequence of \eqref{eq:evapprox}. We summarize the convergence properties of the discrete formulation in \eqref{eq:discrete-form} in the following result.
\begin{theorem}
Let $(\kappa,\u)\in\rrr\times\hh^1(\Omega)$ and $(\kappa_h,\u_h)\in\rrr\times\hh_h$ be an eigenpair of \eqref{eq:weakform} and \eqref{eq:discrete-form} respectively. Under the assumptions listed above, there exist positive constants $C_{\rm ew}$ and $C_{\rm ev}$, depending only on $c,C, \alpha, \lambda, \mu,\kappa$ and $\|p\|_{\infty,\Gamma}$, such that
\begin{align}
    \|\u-\u_h\|_{1,\Omega}\leq\, C_{\rm ew}\delta(\ee(\kappa),\hh_h),
    \quad
    |\kappa-\kappa_h|\leq\, C_{\rm ev}\delta(\ee(\kappa),\hh_h)^2,\label{eq:epapprox}
\end{align}
where the gap $\delta(\cdot,\cdot)$ is defined as in \eqref{eq:gapdefinition}. In addition, the assumption in \eqref{eq:approx} guarantees that $\delta(\ee(\kappa),\hh_h)\to0$ as $h\to0$.
\end{theorem}
We next present some numerical examples to test the theoretical results we showed in the previous sections.

\subsection{Numerical examples}\label{section:numerics}
For all the examples in this section, we utilized Lagrange finite elements of degree $k\geq1$ to approximate the Steklov-Lam\'e eigenpairs. We recall that these finite element spaces have the following interpolation error estimate:
 \begin{align*}
     \|\u-\ii_h\u\|_{s,\Omega} \leq \tilde{C}h^{\min\{k,t\}+1-s}|\u|_{t+1,\Omega},\,\,\forall\,\u\in\hh^{t+1}(\Omega),
 \end{align*}
where the constant $\tilde C>0$ is independent of the meshsize $h$ for shape regular triangulations \cite{ref:girault1986}, $s\geq t>0$, and $|\cdot|_{t+1,\Omega}$ denotes the standard semi-norm in $\hh^t(\Omega)$. We note that this estimate guarantees that the approximation condition in \eqref{eq:approx} is met.

Now let $(\kappa,\u)$ be a Steklov-Lam\'e eigenpair of \eqref{eq:weakform}. Then $\u\in\ee(\kappa)$ and if $\|\u\|_{1,\Omega} = 1$, the interpolation estimate of the Lagrange elements given above with $s=1$ gives
\begin{align*}
    \inf_{\v_h\in\hh_h}\|\u-\v_h\|_{1,\Omega} \leq&\, \|\u-\ii_h\u\|_{1,\Omega}\\
    \leq&\, \tilde{C}h^{\min\{k,t\}}|\u|_{t+1,\Omega}\\
    \leq&\, \tilde{C}h^{\min\{k,t\}}\|\u\|_{1,\Omega}\\
    =&\, \tilde{C}h^{\min\{k,t\}},
\end{align*}
where the last inequality follows from the inequalities $|\u|_{t+1,\Omega}\leq |\u|_{1,\Omega} \leq \|\u\|_{1,\Omega}$. Taking the supremum over $\u\in\ee(\kappa)$ with $\|\u\|_{1,\Omega}$ = 1, we obtain
\begin{align*}
    \delta(\ee(\kappa),\hh_h)\leq \tilde{C}h^{\min\{k,t\}}.
\end{align*}
Using the estimates in \eqref{eq:ewapprox} and \eqref{eq:evapprox} we have the following convergence of the Steklov-Lam\'e eigenpairs
\begin{align}
\|\u-\u_h\|_{1,\Omega}\leq&\, \left(\frac{C\tilde{C}}{\alpha}\right)h^{\min\{k,t\}},\label{eq:cvrateews}\\
    |\kappa-\kappa_h|\leq&\, \left(\frac{C\tilde{C}}{\alpha}\right)^2\big(\max\{\lambda+2\mu,c^2\|p\|_{\infty,\Gamma}\} + (\kappa+1)c^2\|p\|_{\infty,\Gamma}\big)\,h^{2\min\{k,t\}}.\label{eq:cvrateevs}
\end{align}
In all experiments we have used $\pp_1$-conforming elements to compute the approximated eigenpairs on a sequence of regular (not necessarily uniform) meshes. The reference solution was computed with $\pp_1$-conforming elements on a very fine grid. These experiments were implemented in FreeFem++ \cite{ref:freefem}. 

We recall that the rate of convergence of the discrete scheme (cf. \eqref{eq:cvrateews} and \eqref{eq:cvrateevs}) depends entirely on the regularity of the true Steklov-Lam\'e eigenfunctions, and the degree of the local polynomials we choose for our discretization. In fact,
note that the source problem in \eqref{eq:soloperator} is equivalent to the problem of finding $\u\in\hh^1(\Omega)$ such that 
\begin{align*}
    -\bdiv\,\bsigma(\u) = \zero\quad\text{in $\Omega$},\quad \bsigma(\u)\n = \f,\quad\text{on $\Gamma$},
\end{align*}
for a given $\f\in\ll^2(\Omega)$.
Since the boundary condition of the problem above is of Neumann type,
in the presence of corners or edges on the boundary, the Steklov-Lam\'e eigenfunctions belong to $\hh^{1+s}(\Omega)$ for all $s\in(0,r_1]$, where $r_1$ is the first positive root of the following nonlinear equation \cite{ref:grisvard1989,ref:nicaise1992}
\begin{align}
    r^2\sin^2(\theta) = \sin^2(r\cdot\theta),\quad r\in\rrr,\label{eq:regularity-eqs}
\end{align}
with $\theta$ representing the largest interior angle of $\Omega$. Note that $r_1 = 1$ is always a solution of \eqref{eq:regularity-eqs}. Thus, the best possible space for the solutions of  \eqref{eq:gen-steklovlame} in the presence of corners or edges in the domain is $\hh^2(\Omega)$. 
We notice that in the case of Neumann boundary conditions, the regularity of the eigenvectors does not seem to be affected by the Lam\'e parameters (see, e.g. \cite[Theorem 2.1]{ref:nicaise1992}).

\autoref{table:cvsquare} shows the computed rate of convergence of the Steklov-Lam\'e eigenvectors on the unit square $(0,1)^2$. Since the unit square is a convex polygon, we expect that the eigenfunctions belong to $\hh^2(\Omega)$. We see that all computed rates of convergence are around 2. As seeing in \autoref{table:cvcircle}, the rate of convergence possesses a similar behaviour on the unit disk to that on the unit square. We can see that all computed rates of convergence are slightly above 2. We comment that, even though the circle does not have any corners, the reference solution was computed on a computational domain $\Omega_h$ representing a polygon with corners being prescribed on the boundary of the unit circle. Thus, the eigenfunctions on $\Omega_h$ were defined on a convex polygon. This means that the reference eigenfunctions belong to $\hh^2(\Omega_h)$. We also present the convergence of the lowest Steklov-Lam\'e eigenvalues on the L-shape domain $(-1,1)^2\backslash[0,1)^2$. The singularities of the eigenfunctions on the L-shape domain are ruled by the largest interior angle. Solving the nonlinear equation in \eqref{eq:regularity-eqs} with $\theta = \frac{3}{2}\pi$ we obtain $r_1 = 0.5445$. We can see from \autoref{table:cvlshaped} that all computed rates of convergence are above the lowest theoretical convergence rate of $2r_1 = 1.0890$. The last numerical example concerns the convergence of the first 7 non-zero Steklov-Lam\'e eigenvalues on the unit cube $(0,1)^3$. We see in \autoref{table:cvcube} that all computed rated of convergence are around 2. This is expected as the cube is a convex polyhedral domain.


\graphicspath{{./images/}}
\begin{figure}[ht!]
\centering\includegraphics[width = 0.5\textwidth, 
height=0.25\textheight]{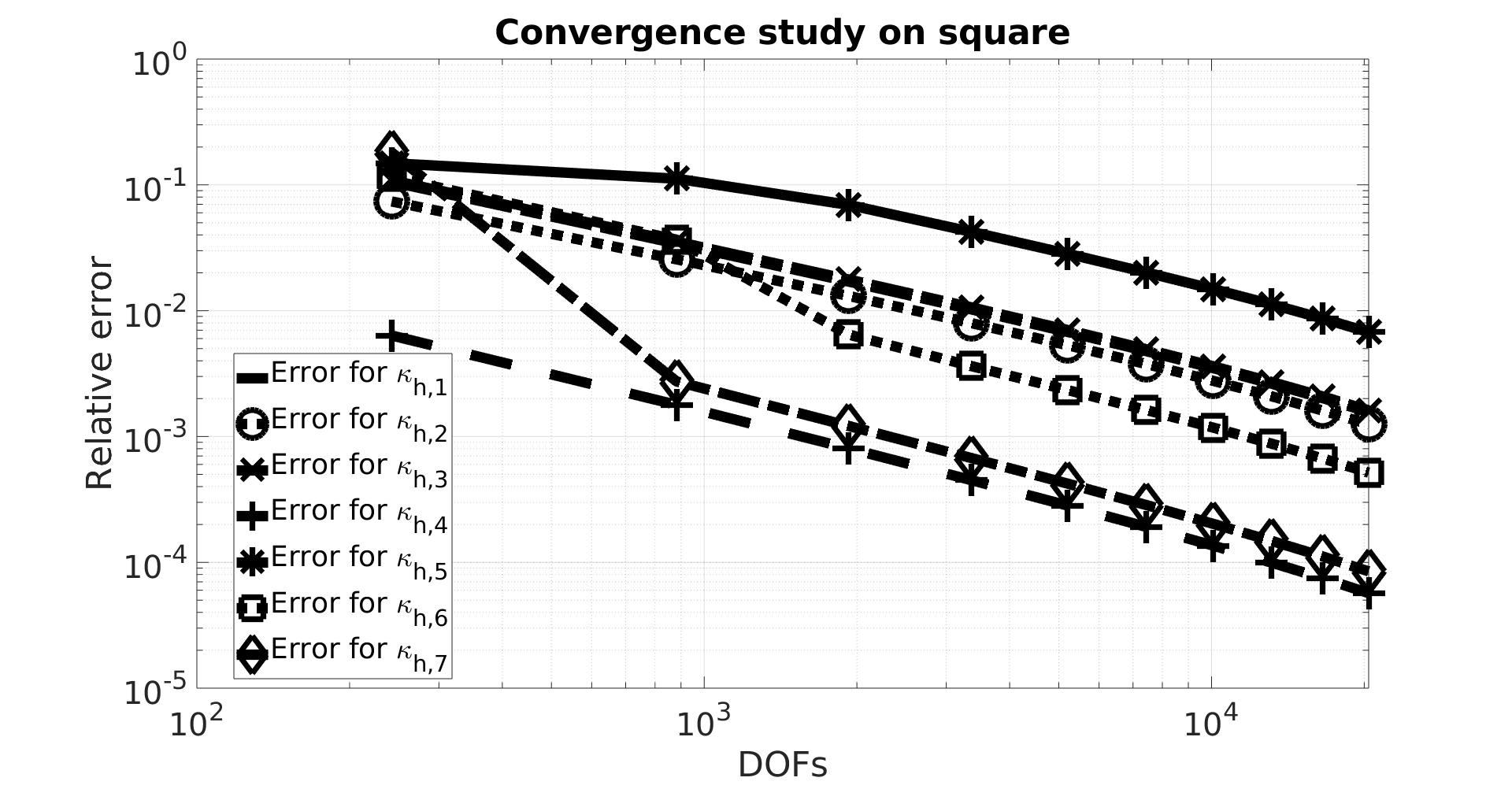}\includegraphics[width = 0.5\textwidth, 
height=0.25\textheight]{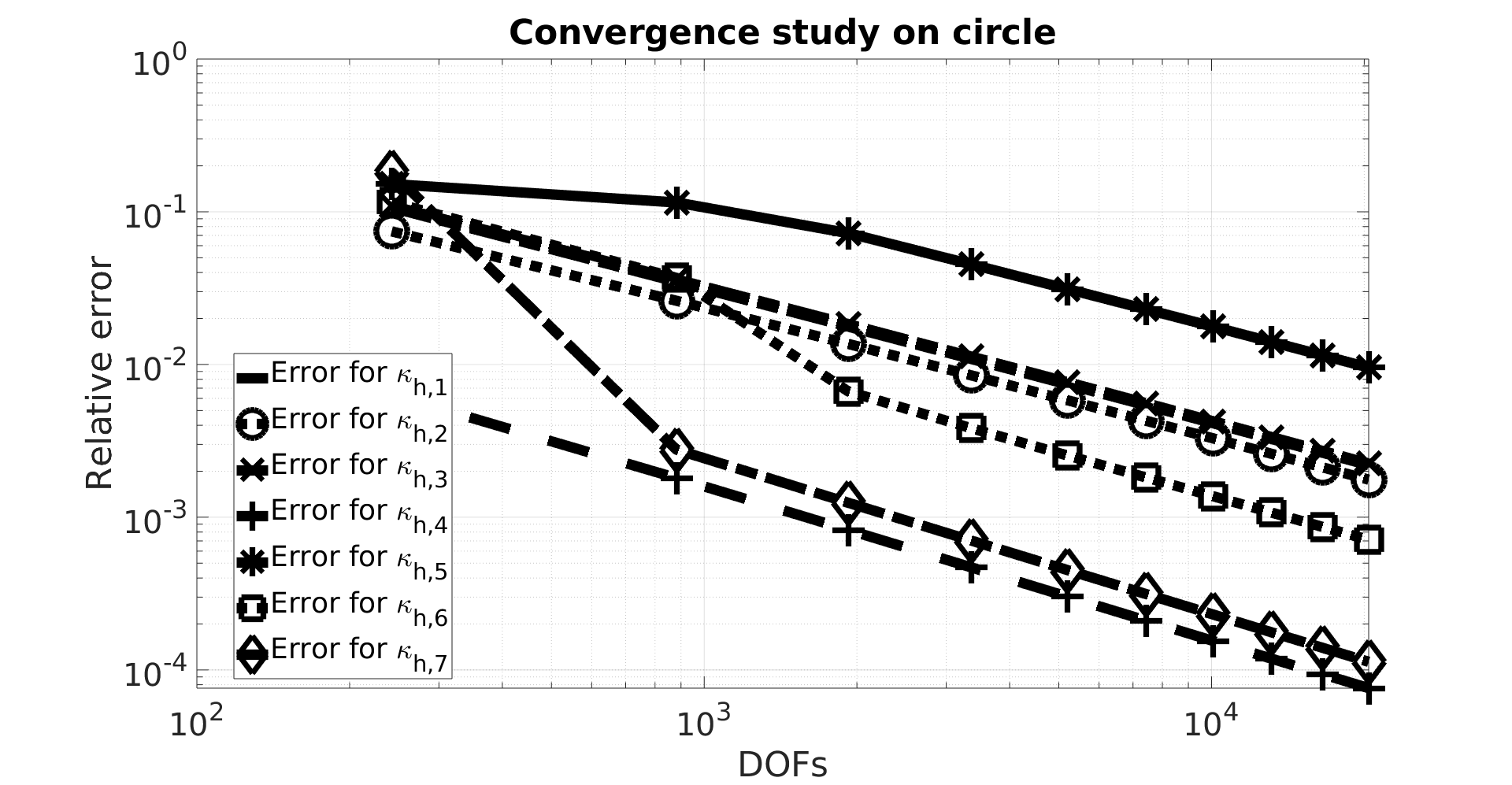}
\centering\includegraphics[width = 0.5\textwidth, 
height=0.25\textheight]{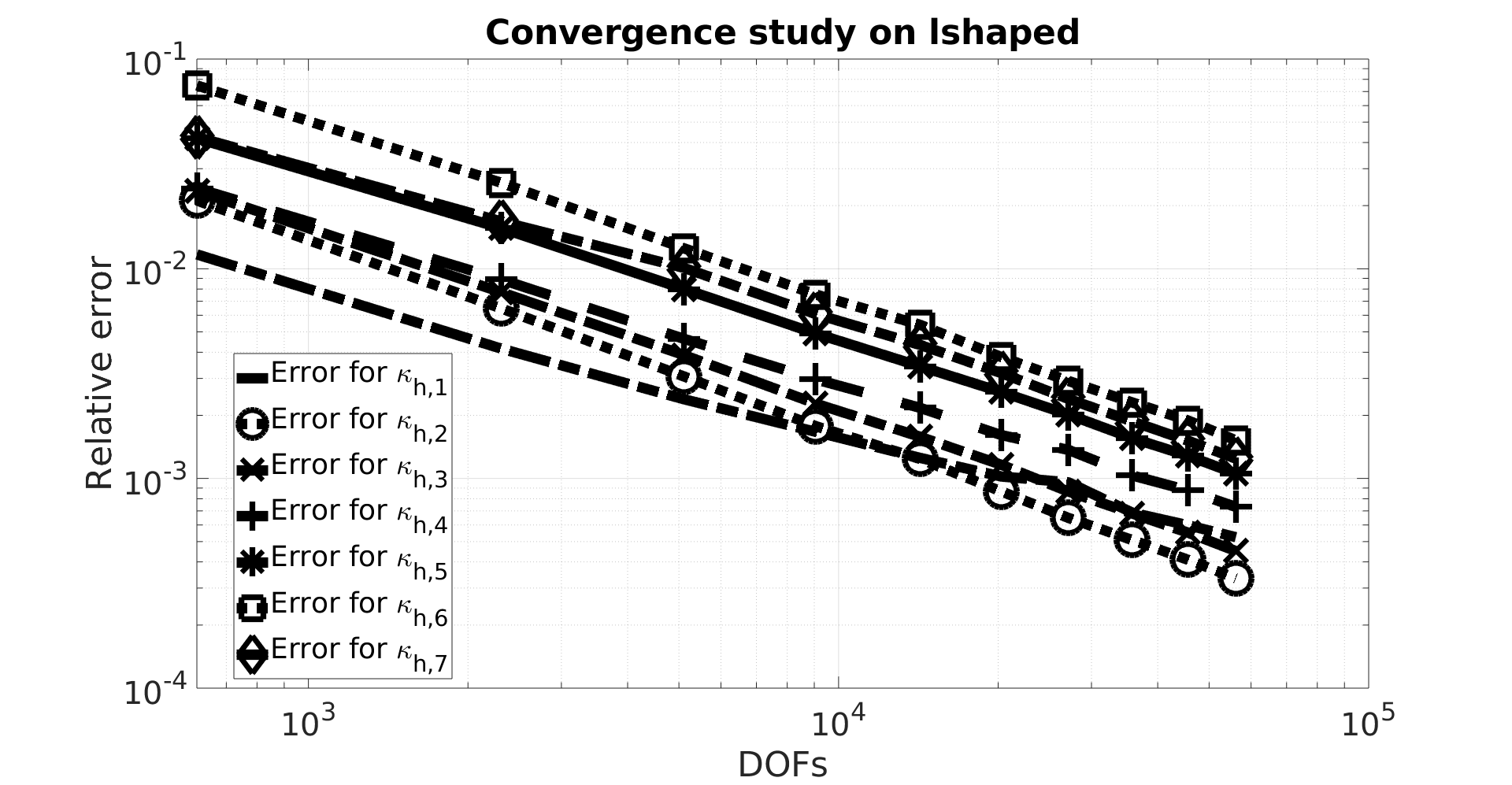}\includegraphics[width = 0.5\textwidth, 
height=0.25\textheight]{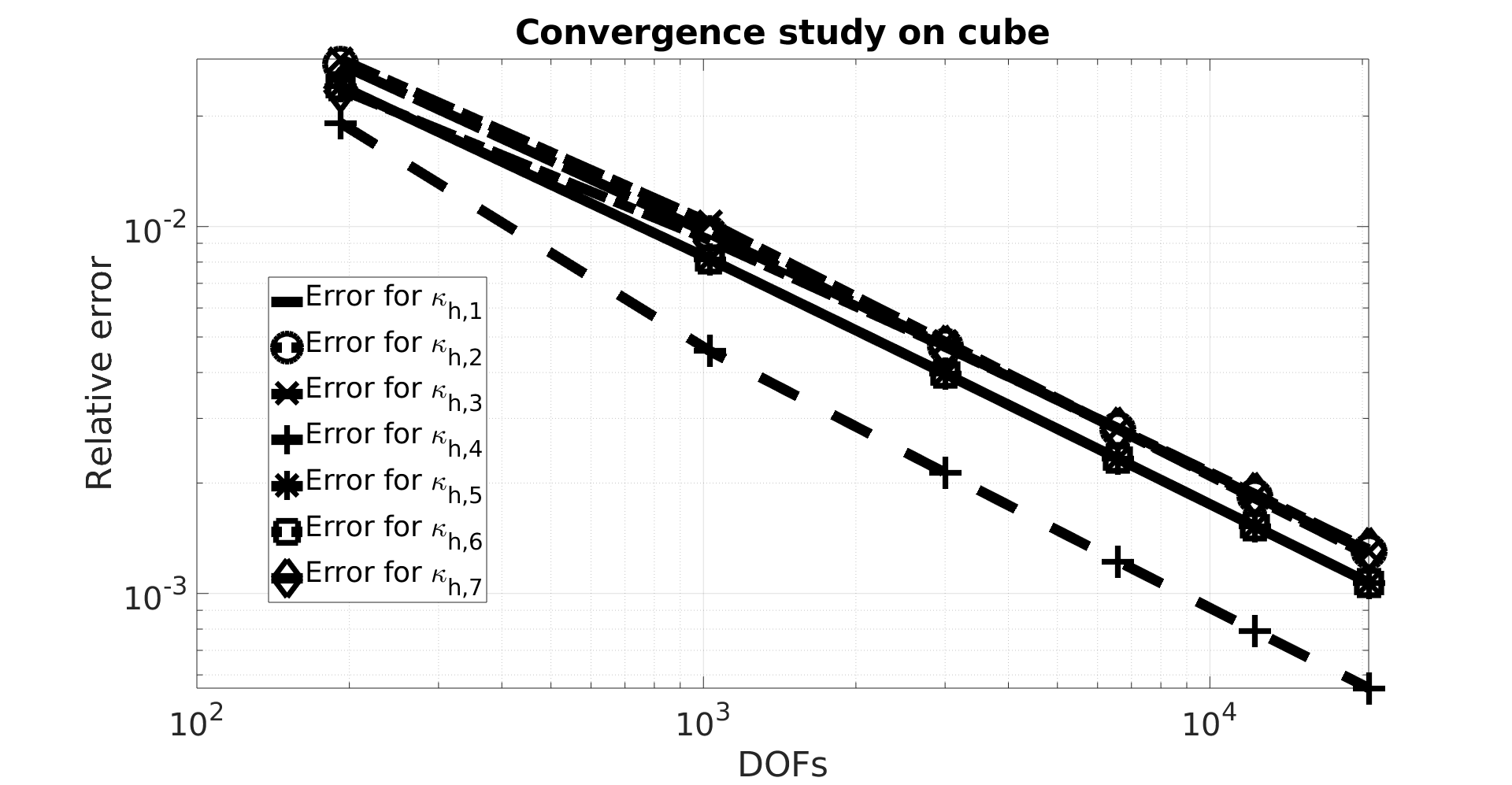}
\caption{Convergence study for the first 6 non-zero eigenvalues on the unit square (top-left), the unit disk (top-right), the L-shape (bottom-left) and the unit cube (bottom-right). Recall that the number of degrees of freedom $N$ and the meshsize $h$ scale as $h = \oooo(N^{-1/d})$.
}\label{fig:cvstudy}
\end{figure}

The convergence history of the first 7 computed Steklov-Lam\'e eigenvalues is shown in \autoref{fig:cvstudy} for the unit square, unit disk, L-shape domain, and unit cube. These correspond to the same computed eigenvalues to those shown in \autoref{table:cvsquare}, \autoref{table:cvcircle}, \autoref{table:cvlshaped}, and \autoref{table:cvcube}.

\begin{table}[ht!]
    \centering
    {\footnotesize{\begin{tabular}{c|c|c|c|c|c|c|c}
       EVs &$N = $ 242&$N = $ 1922&$N = $ 5202&$N = $ 10082&$N = $ 16562& Reference& Rate\\\hline
$\kappa_{h,1}$ &2.800192&2.57581&2.549729&2.541415&2.537678& 2.532570 & 2.0419\\
$\kappa_{h,2}$ &2.872823&2.710273&2.689398&2.682579&2.679477& 2.675175 &    2.0136\\
$\kappa_{h,3}$ &2.966591&2.722965&2.69431&2.685177&2.681081& 2.675513 &     2.0499\\
$\kappa_{h,4}$ &3.734775&3.714195&3.712252&3.711705&3.711479& 3.711202 &    2.2488\\
$\kappa_{h,5}$ &5.480897&5.103026&4.906772&4.842315&4.81281& 4.771482 &     1.9973\\
$\kappa_{h,6}$ &5.860259&5.288715&5.266997&5.260878&5.258216& 5.254700 &    2.1177\\
$\kappa_{h,7}$ &6.84993&5.806006&5.801406&5.800129&5.799602& 5.79879 & 
2.2558
    \end{tabular}}}
    \caption{First 7 (counted with their multiplicities) computed nonzero Steklov-Lam\'e eigenvalues on the unit square with Lam\'e parameters $\mu = \lambda = p = 1$. Recall that $\kappa_h = w_h +1$.}
    \label{table:cvsquare}
\end{table}

\begin{table}[ht!]
    \centering
    {\footnotesize{\begin{tabular}{c|c|c|c|c|c|c|c}
      EVs &$N = $ 190&$N = $ 1520&$N = $ 4046&$N =  $ 7794&$N =  $ 12956& Reference & Rate\\\hline
      $\kappa_{h,1}$ & 3.003639&3.000406&3.000146&3.000075&3.000045& 3.000009 & 2.1645\\
      $\kappa_{h,2}$ & 
3.003639&3.000406&3.000146&3.000075&3.000045& 3.000009 & 2.0699\\
$\kappa_{h,3}$ &
3.059373&3.006645&3.002431&3.001284&3.000756 & 3.000009 & 2.1528\\
$\kappa_{h,4}$ &
3.063728&3.00676&3.00252&3.001312&3.00077& 3.000009 & 2.0730\\
$\kappa_{h,5}$ &
4.324313&4.036279&4.013165&4.006795&4.004003& 4.000014 & 2.1638\\
$\kappa_{h,6}$ &
4.39301&4.03863&4.014033&4.007199&4.004586& 4.000014 & 2.1349
\\
$\kappa_{h,7}$ &
5.007277&5.000812&5.000292&5.000149&5.00009& 5.000018 & 2.1995
    \end{tabular}}}
    \caption{First 7 (counted with their multiplicities) computed nonzero Steklov-Lam\'e eigenvalues on the unit disk with Lam\'e parameters $\mu = \lambda = p = 1$. Recall that $\kappa_h = w_h +1$.}
    \label{table:cvcircle}
\end{table}

\begin{table}[ht!]
    \centering
    {\footnotesize{\begin{tabular}{c|c|c|c|c|c|c|c}
     EVs &$N = $ 616&$N = $ 5114&$N = $ 14244&$N = $ 27164&$N =  $ 45620& Reference & Rate\\\hline
      $\kappa_{h,1}$ & 1.168833&1.158064&1.156757&1.156416&1.156000 & 1.155308 & 1.5808\\
$\kappa_{h,2}$ &1.750674&1.719661&1.716536&1.715522&1.715113 & 1.714410 & 1.6349\\ 
$\kappa_{h,3}$ &2.061&2.021514&2.016901&2.015461&2.014828  & 2.01371 & 1.6895\\
$\kappa_{h,4}$ &2.177396&2.135806&2.130581&2.128869&2.127828 & 2.125962 & 1.4848\\
$\kappa_{h,5}$ &2.724265&2.635748&2.623772&2.620085&2.618166 & 2.614815 & 1.7108
\\
$\kappa_{h,6}$ &2.94148&2.770783&2.751355&2.744538&2.741737 & 2.736563 & 1.5737
\\
$\kappa_{h,7}$ &3.513536&3.404443&3.385&3.378496&3.375532 &3.370429 & 1.8972
    \end{tabular}}}
    \caption{First 7 (counted with their multiplicities) computed nonzero Steklov-Lam\'e eigenvalues on the L-shape with Lam\'e parameters $\mu = \lambda = p = 1$. Recall that $\kappa_h = w_h +1$.}
    \label{table:cvlshaped}
\end{table}

\begin{table}[ht!]
    \centering
    {\footnotesize{\begin{tabular}{c|c|c|c|c|c|c|c}
     EVs &$N=$ 1029&$N=$ 3000&$N=$ 6591&$N=$ 12288&$N=$ 20577& Reference & Rate\\\hline
$\kappa_{h,1}$ & 2.082904&2.072949&2.068958&2.066928&2.065802 & 2.06318 & 2.0820 \\
$\kappa_{h,2}$ & 2.082904&2.072949&2.068958&2.066979&2.065861 & 2.063182 & 2.0238 \\ 
$\kappa_{h,3}$ & 2.084313&2.073188&2.068959&2.066979&2.065861 & 2.063182 & 2.0242 \\
$\kappa_{h,4}$ & 2.099347&2.094227&2.092327&2.091424&2.090926 & 2.089772 & 2.0952 \\
$\kappa_{h,5}$ & 2.106801&2.098076&2.094638&2.092955&2.092011 & 2.089774 & 2.0469 \\
$\kappa_{h,6}$ & 2.106801&2.098076&2.094638&2.092955&2.092011 & 2.089774 & 2.0469 \\
$\kappa_{h,7}$ & 2.119688&2.110279&2.106316&2.104312&2.103168 & 2.100399 & 2.0007
    \end{tabular}}}
    \caption{First 7 (counted with their multiplicities) computed nonzero Steklov-Lam\'e eigenvalues on the unit cube with Lam\'e parameters $\mu = \lambda = p = 1$. Recall that $\kappa_h = w_h +1$.}
    \label{table:cvcube}
\end{table}

\graphicspath{{./images/SteklovLame/square/}}
\begin{figure}[!ht]
\centering\includegraphics[width = .3\textwidth, 
height = .125\textheight]{EW3-xcomp.png}\includegraphics[width = .3\textwidth, 
height = .125\textheight]{EW3-ycomp.png}
\centering\includegraphics[width = .3\textwidth, 
height = .125\textheight]{EW4-xcomp.png}\includegraphics[width = .3\textwidth, 
height = .125\textheight]{EW4-ycomp.png}
\centering\includegraphics[width = .3\textwidth, 
height = .125\textheight]{EW5-xcomp.png}\includegraphics[width = .3\textwidth, 
height = .125\textheight]{EW5-ycomp.png}
\centering\includegraphics[width = .3\textwidth, 
height = .125\textheight]{EW6-xcomp.png}\includegraphics[width = .3\textwidth, 
height = .125\textheight]{EW6-ycomp.png}
\centering\includegraphics[width = .3\textwidth, 
height = .125\textheight]{EW7-xcomp.png}\includegraphics[width = .3\textwidth, 
height = .125\textheight]{EW7-ycomp.png}
\centering\includegraphics[width = .3\textwidth, 
height = .125\textheight]{EW8-xcomp.png}\includegraphics[width = .3\textwidth, 
height = .125\textheight]{EW8-ycomp.png}
\centering\includegraphics[width = .3\textwidth, 
height = .125\textheight]{EW9-xcomp.png}\includegraphics[width = .3\textwidth, 
height = .125\textheight]{EW9-ycomp.png}
\caption{First 7 eigenfunctions $\u_h$ on the unit square associated to the first 7 non-zero eigenvalues $\kappa_h$ (counted with their multiplicities). The x-component of $\u_h$ is on the left column while the y-component of $\u_h$ is shown on the left column.}\label{fig:square}
\end{figure}



\graphicspath{{./images/SteklovLame/circle/}}
\begin{figure}[!ht]
\centering\includegraphics[width = .3\textwidth, 
height = .125\textheight]{EW3-xcomp.png}\includegraphics[width = .3\textwidth, 
height = .125\textheight]{EW3-ycomp.png}
\centering\includegraphics[width = .3\textwidth, 
height = .125\textheight]{EW4-xcomp.png}\includegraphics[width = .3\textwidth, 
height = .125\textheight]{EW4-ycomp.png}
\centering\includegraphics[width = .3\textwidth, 
height = .125\textheight]{EW5-xcomp.png}\includegraphics[width = .3\textwidth, 
height = .125\textheight]{EW5-ycomp.png}
\centering\includegraphics[width = .3\textwidth, 
height = .125\textheight]{EW6-xcomp.png}\includegraphics[width = .3\textwidth, 
height = .125\textheight]{EW6-ycomp.png}
\centering\includegraphics[width = .3\textwidth, 
height = .125\textheight]{EW7-xcomp.png}\includegraphics[width = .3\textwidth, 
height = .125\textheight]{EW7-ycomp.png}
\centering\includegraphics[width = .3\textwidth, 
height = .125\textheight]{EW8-xcomp.png}\includegraphics[width = .3\textwidth, 
height = .125\textheight]{EW8-ycomp.png}
\centering\includegraphics[width = .3\textwidth, 
height = .125\textheight]{EW9-xcomp.png}\includegraphics[width = .3\textwidth, 
height = .125\textheight]{EW9-ycomp.png}
\caption{First 7 eigenfunctions $\u_h$ on the unit disk associated to the first 7 non-zero eigenvalues $\kappa_h$ (counted with their multiplicities). The x-component of $\u_h$ is on the left column while the y-component of $\u_h$ is shown on the left column.}\label{fig:circle}
\end{figure}

\graphicspath{{./images/SteklovLame/lshaped/}}
\begin{figure}[!ht]
\centering\includegraphics[width = .3\textwidth, 
height = .125\textheight]{EW3-xcomp.png}\includegraphics[width = .3\textwidth, 
height = .125\textheight]{EW3-ycomp.png}
\centering\includegraphics[width = .3\textwidth, 
height = .125\textheight]{EW4-xcomp.png}\includegraphics[width = .3\textwidth, 
height = .125\textheight]{EW4-ycomp.png}
\centering\includegraphics[width = .3\textwidth, 
height = .125\textheight]{EW5-xcomp.png}\includegraphics[width = .3\textwidth, 
height = .125\textheight]{EW5-ycomp.png}
\centering\includegraphics[width = .3\textwidth, 
height = .125\textheight]{EW6-xcomp.png}\includegraphics[width = .3\textwidth, 
height = .125\textheight]{EW6-ycomp.png}
\centering\includegraphics[width = .3\textwidth, 
height = .125\textheight]{EW7-xcomp.png}\includegraphics[width = .3\textwidth, 
height = .125\textheight]{EW7-ycomp.png}
\centering\includegraphics[width = .3\textwidth, 
height = .125\textheight]{EW8-xcomp.png}\includegraphics[width = .3\textwidth, 
height = .125\textheight]{EW8-ycomp.png}
\centering\includegraphics[width = .3\textwidth, 
height = .125\textheight]{EW9-xcomp.png}\includegraphics[width = .3\textwidth, 
height = .125\textheight]{EW9-ycomp.png}
\caption{First 7 eigenfunctions $\u_h$ on the L-shaped $(-1,1)^2\backslash[0,1)^2$ associated to the first 7 non-zero eigenvalues $\kappa_h$ (counted with their multiplicities). The x-component of $\u_h$ is on the left column while the y-component of $\u_h$ is shown on the left column.}\label{fig:lshaped}
\end{figure}

\section{Conclusions}
In this paper we presented a study of Steklov eigenvalues for the Lam\'e operator in linear elasticity. We established the existence of a countable spectrum for the Steklov-Lam\'e eigenproblem. We also proved a different version of Korn's inequality (cf. \eqref{eq:extendedkorns}). This inequality was used to obtain the unique solvability of the source problem used to define the solution operator (cf. \eqref{eq:soloperator}). A spectral characterization was given, showing that the kernel of the solution operator is $\hh^1_0(\Omega)$. The compactness of a modified solution operator was achieved by using the continuity and compactness of the trace operator.

In addition, a conforming discrete formulation was proposed in \ref{section:discrete}. Based on the theory developed in \cite{ref:babuskaosborn1991} we were able to show that the discrete formulation in \eqref{eq:discrete-form} provides the correct approximation to the true Steklov-Lam\'e eigenpairs. Finally, we provided numerical results showing the convergence properties of the discrete scheme. We showed that the rate of convergence are close to those predicted by the regularity of the eigenfunctions.

Many questions regarding the Steklov-Lam\'e spectrum are still to be answered. For example, it is known that primal formulations are not suitable to deal with situations in which $\lambda\to+\infty$ (for incompressible materials). In fact, we can see that the bilinear form $a$ becomes unbounded in such cases. To remedy this issue, one can consider a mixed formulation of \eqref{eq:steklovlame}, following the ideas given in, for instance \cite{ref:meddahi2013}. 

On the other hand, no true solutions of the Steklov-Lam\'e eigenproblem are yet known. The numerical examples on the unit disk presented in \ref{section:numerics} show that for $\mu = \lambda = p = 1$, the eigenvalues of \eqref{eq:steklov-lame} are $w_0 = 0$ with multiplicity 3, $w_{2k-1} = 2k$ with multiplicity 4, and $w_{2k} = 2k+1$ with multiplicity 2, for all $k\in\nnn$. A similar behaviour is exhibited by the Steklov eigenvalues of the Laplacian where all positive eigenvalues have multiplicity 2 (see, e.g. \cite{ref:girouard2017}).


\section*{Acknowledgements} Sebasti\'an Dom\'inguez thanks the financial support of the Pacific Institute for the Mathematical Sciences. The author acknowledges Nilima Nigam (Simon Fraser University) and Richard Laugesen (University of Illinois Urbana-Champaign). Many ideas presented in this manuscript were discussed in length with them.

%
%
%
%
%
%
%
%
%
%
%

\bibliographystyle{plain}
\bibliography{ThesisReferences.bib}

\begin{thebibliography}{10}

\bibitem{ref:duran2006}
G.~Acosta, R.~G. Dur{\'a}n, and M.~A. Muschietti.
\newblock {Solutions of the divergence operator on John domains}.
\newblock {\em Advances in Mathematics}, 206(2):373--401, 2006.

\bibitem{ref:atkinson1964}
F.~V. Atkinson.
\newblock {\em Discrete and Continuous Boundary Problems}.
\newblock ISSN. Elsevier Science, 1964.

\bibitem{ref:babuskaosborn1991}
I.~Babuska and J.~Osborn.
\newblock Eigenvalue problems.
\newblock In {\em Finite Element Methods (Part 1)}, volume~2 of {\em Handbook
  of Numerical Analysis}, pages 641--787. Elsevier, 1991.

\bibitem{ref:biegert2009}
M.~Biegert.
\newblock {On traces of Sobolev functions on the boundary of extension
  domains}.
\newblock {\em Proceedings of the American Mathematical Society},
  137(12):4169--4176, 2009.

\bibitem{ref:brenner2003}
S.~C. Brenner.
\newblock {Korn's Inequalities for Piecewise {$H^1$} Vector Fields}.
\newblock {\em Mathematics of Computation}, 73(247):1067--1087, 2004.

\bibitem{ref:brewster2014}
K.~Brewster, D.~Mitrea, I.~Mitrea, and M.~Mitrea.
\newblock Extending sobolev functions with partially vanishing traces from
  locally ($\varepsilon$, $\delta$)-domains and applications to mixed boundary
  problems.
\newblock {\em Journal of Functional Analysis}, 266(7):4314--4421, 2014.

\bibitem{ref:cao2013}
L.~Cao, L.~Zhang, W.~Allegretto, and Y.~Lin.
\newblock {Multiscale asymptotic method for Steklov eigenvalue equations in
  composite media}.
\newblock {\em SIAM J. Numer. Anal.}, 51:273--296, 2013.

\bibitem{ref:damlamian2018}
A.~Damlamian.
\newblock {Some Remarks on Korn Inequalities}.
\newblock {\em Chinese Annals of Mathematics, Series B}, 39(2):335--344, Mar
  2018.

\bibitem{ref:demengel2012}
F.~Demengel and G.~Demengel.
\newblock {\em Functional Spaces for the Theory of Elliptic Partial
  Differential Equations}.
\newblock Universitext. Springer London, 2012.

\bibitem{ref:descloux1978-1}
J.~Descloux, N.~Nassif, and J.~Rappaz.
\newblock On spectral approximation. part 1. the problem of convergence.
\newblock {\em ESAIM: Mathematical Modelling and Numerical
  Analysis-Mod{\'e}lisation Math{\'e}matique et Analyse Num{\'e}rique},
  12(2):97--112, 1978.

\bibitem{ref:duran2004}
R.~G. Dur{\'a}n and M.~A. Muschietti.
\newblock {The Korn inequality for Jones domains}.
\newblock {\em Electronic Journal of Differential Equations (EJDE)[electronic
  only]}, 2004(127):1--10, 2004.

\bibitem{ref:gatica2014}
G.~N. Gatica.
\newblock {\em A simple introduction to the mixed finite element method}.
\newblock SpringerBriefs in Mathematics. Springer International Publishing, 1st
  edition, 2014.

\bibitem{ref:girault1986}
V.~Girault and P.~A. Raviart.
\newblock {\em {Finite element methods for Navier-Stokes equations: theory and
  algorithms}}.
\newblock Springer series in computational mathematics. Springer-Verlag, 1986.

\bibitem{ref:girouard2017}
A.~Girouard and I.~Polterovich.
\newblock {Spectral geometry of the Steklov problem (survey article)}.
\newblock {\em Journal of Spectral Theory}, 7(2):321--360, 2017.

\bibitem{ref:gomez2018}
D.~G{\'o}mez, S.~A. Nazarov, and M.~E. P{\'e}rez.
\newblock {Homogenization of Winkler--Steklov spectral conditions in
  three-dimensional linear elasticity}.
\newblock {\em Zeitschrift f{\"u}r angewandte Mathematik und Physik}, 69(2):35,
  Feb 2018.

\bibitem{ref:grisvard1989}
P.~Grisvard.
\newblock Singularit\'{e}s en elasticit\'{e}.
\newblock {\em Arch. Rational Mech. Anal.}, 107(2):157--180, 1989.

\bibitem{ref:freefem}
F.~Hecht.
\newblock New development in freefem++.
\newblock {\em Journal of numerical mathematics}, 20(3-4):251--266, 2012.

\bibitem{ref:hinton1990}
D.~B. Hinton and J.~K. Shaw.
\newblock Differential operators with spectral parameter incompletely in the
  boundary conditions.
\newblock {\em Funkcialaj Ekvacioj}, 33:363--385, 1990.

\bibitem{ref:ionescu2005}
I.~Ionescu, D.~Onofrei, and B.~Vernescu.
\newblock {$\Gamma$-convergence for a fault model with slip-weakening friction
  and periodic barriers}.
\newblock {\em Quarterly of applied mathematics}, 63(4):747--778, 2005.

\bibitem{ref:ionescu1996}
I.~Ionescu and J.-C. Paumier.
\newblock On the contact problem with slip displacement dependent friction in
  elastostatics.
\newblock {\em International journal of engineering science}, 34(4):471--491,
  1996.

\bibitem{ref:john1961}
F.~John.
\newblock Rotation and strain.
\newblock {\em Communications on Pure and Applied Mathematics}, 14(3):391--413,
  1961.

\bibitem{ref:jones1981}
P.~W. Jones.
\newblock Quasiconformal mappings and extendability of functions in sobolev
  spaces.
\newblock {\em Acta Math.}, 147:71--88, 1981.

\bibitem{ref:kikuchi1988}
N.~Kikuchi and J.~T. Oden.
\newblock {\em Contact problems in elasticity: a study of variational
  inequalities and finite element methods}, volume~8.
\newblock siam, 1988.

\bibitem{ref:korn1906}
A.~Korn.
\newblock Die eigenschwingungen eines elastichen korpers mit ruhender
  oberflache.
\newblock {\em Akademie der Wissensch Munich, Math-phys. Kl, Beritche},
  36:351--401, 1906.

\bibitem{ref:korn1909}
A.~Korn.
\newblock Ubereinige ungleichungen, welche in der theorie der elastischen und
  elektrischen schwingungen eine rolle spielen.
\newblock {\em Bulletin Internationale, Cracovie Akademie Umiejet, Classe de
  sciences mathematiques et naturelles}, pages 705--724, 1909.

\bibitem{ref:lopez2018}
F.~L\'{o}pez-Garc\'{i}a.
\newblock Weighted {K}orn inequalities on {J}ohn domains.
\newblock {\em Studia Math.}, 241(1):17--39, 2018.

\bibitem{ref:martio1978}
O.~Martio and J.~Sarvas.
\newblock Injectivity theorems in plane and space.
\newblock {\em Annales academiae scientiarum fennicae}, 4:383--401, 1978/1979.

\bibitem{ref:mayer2012}
H.~C. Mayer and R.~Krechetnikov.
\newblock Walking with coffee: Why does it spill?
\newblock {\em Physical Review E}, 85(4):046117, 2012.

\bibitem{ref:mclean2000}
W.~McLean.
\newblock {\em Strongly Elliptic Systems and Boundary Integral Equations}.
\newblock Cambridge University Press, 2000.

\bibitem{ref:meddahi2013}
S.~Meddahi, , D.~Mora, and R.~Rodr{\'i}guez.
\newblock Finite element spectral analysis for the mixed formulation of the
  elasticity equations.
\newblock {\em SIAM Journal on Numerical Analysis}, 51(2):1041--1063, 2013.

\bibitem{ref:necas2011}
J.~Necas.
\newblock {\em Direct methods in the theory of elliptic equations}.
\newblock Springer Science \& Business Media, 2011.

\bibitem{ref:nicaise1992}
S.~Nicaise.
\newblock About the {L}am\'{e} system in a polygonal or a polyhedral domain and
  a coupled problem between the {L}am\'{e} system and the plate equation. {I}.
  {R}egularity of the solutions.
\newblock {\em Ann. Scuola Norm. Sup. Pisa Cl. Sci. (4)}, 19(3):327--361, 1992.

\bibitem{ref:nitsche1981}
J.~A. Nitsche.
\newblock {On Korn's second inequality}.
\newblock {\em ESAIM: Mathematical Modelling and Numerical Analysis -
  Mod{\'e}lisation Math{\'e}matique et Analyse Num{\'e}rique}, 15(3):237--248,
  1981.

\bibitem{ref:raviart1983}
P.~A. Raviart and J.~M. Thomas.
\newblock {\em Introduction {\`a} l'analyse num{\'e}rique des {\'e}quations aux
  d{\'e}riv{\'e}es partielles}.
\newblock Collection Math{\'e}matiques appliqu{\'e}es pour la ma{\^i}trise.
  Masson, Paris, 1983.

\bibitem{ref:steklov1902}
W.~Stekloff.
\newblock Sur les probl{\`e}mes fondamentaux de la physique math{\'e}matique.
\newblock In {\em Annales scientifiques de l'{\'E}cole Normale Sup{\'e}rieure},
  volume~19, pages 191--259, 1902.

\bibitem{ref:ting1972}
T.~W. Ting.
\newblock {Generalized Korn’s inequalities}.
\newblock {\em Tensor}, 25:295--302, 1972.

\end{thebibliography}

\end{document}